\documentclass[11pt,reqno]{amsart}
\usepackage{amsfonts}
\usepackage{indentfirst}
\usepackage{graphicx}
\usepackage{amssymb}
\usepackage{color,xcolor}
\usepackage{epstopdf}
\usepackage{graphicx}
\usepackage{epsfig}
\usepackage{caption}
\usepackage{subfigure}
\usepackage{mathrsfs}
\usepackage{bbm}
\usepackage{chngpage}
\usepackage{cite}
\usepackage{multirow}
\usepackage{multicol}
\usepackage{dsfont}
\usepackage{bm}
\usepackage{amssymb,amsmath,amsthm,mathrsfs}%
\usepackage{exscale}
\usepackage{tikz-cd}
\usepackage{relsize}
\usepackage{amssymb}
\usepackage{hyperref}
\usepackage{url}
\usepackage{tikz-cd}
\usepackage{comment}
\usepackage[misc,geometry]{ifsym} 

\allowdisplaybreaks

\hypersetup{hypertex=green,
	colorlinks=true,
	linkcolor=green,
	anchorcolor=green,
	citecolor=red}

\theoremstyle{definition}
\newtheorem{theorem}{Theorem}[section]

\newtheorem{example}{Example}

\newtheorem{remark}{Remark}[section]
\newtheorem{lemma}{Lemma}[section]

\numberwithin{equation}{section}%
\numberwithin{table}{section}%
\numberwithin{figure}{section}

\def\3bar{{|\hspace{-.02in}|\hspace{-.02in}|}}

\newcommand\grad{\operatorname{grad}}
\renewcommand\div{\operatorname{div}}

\newcommand\curl{\operatorname{curl}}

\newcommand\Span{\operatorname{span}}

\newcommand\x{\times}

\def\d{\text{d}}

\begin{document}
	\title[Fully $H$(GRADCURL)-nonconforming finite elements]{Fully $H$(GRADCURL)-nonconforming finite element method for the singularly perturbed quad-curl problem on cubical meshes}


\author{Lixiu Wang}
\email{lxwang@ustb.edu.cn}
\address{School of Mathematics and Physics, University of Science and Technology Beijing, Beijing 100083, China.
}\thanks{This work is supported in part by the National Natural Science Foundation of China grants NSFC 12101036 and the Fundamental Research Funds for the Central Universities FRF-TP-22-096A1.}
\author{Mingyan Zhang}
\email{myzhang@csrc.ac.cn.}
 \address{Corresponding author. Beijing Computational Science Research Center, Beijing 100193, China.}
 \author{Qian Zhang}
 \email{qzhang15@mtu.edu}
 \address{Department of Mathematical Sciences, Michigan Technological University, Houghton, MI 49931, USA. }
	%


	\begin{abstract}
 In this paper, we develop two fully nonconforming (both  $H(\grad\curl)$-nonconforming and  $H(\curl)$-nonconforming) finite elements on cubical meshes which can fit into the Stokes complex. The newly proposed elements have 24 and 36 degrees of freedom, respectively. Different from the fully $H(\grad\curl)$-nonconforming tetrahedral finite elements in \cite{Huang2020Nonconforming}, the elements in this paper lead to a robust finite element method to solve the singularly perturbed quad-curl problem. To confirm this,  we prove the optimal convergence of order $\mathcal O(h)$ for a fixed parameter $\epsilon$ and the uniform convergence of order $\mathcal O(h^{1/2})$ for any value of $\epsilon$. Some numerical examples are used to verify the correctness of the theoretical analysis.  
	\end{abstract}	
	\maketitle
	\section{Introduction}
The quad-curl equation appears in various models, such as the inverse electromagnetic scattering theory \cite{Cakoni2017A,Monk2012Finite, Sun2016A}, couple stress theory in linear elasticity \cite{mindlin1962effects,park2008variational}, and  magnetohydrodynamics \cite{Zheng2011A}. The corresponding quad-curl eigenvalue problem plays a fundamental role in the analysis and computation of the electromagnetic interior transmission eigenvalues \cite{sun2011iterative}. 
In this paper, we consider the following quad-curl singular perturbation problem on a bounded polyhedral domain $\Omega\subset\mathbb{R}^3$: find $\bm u$ such that 
	\begin{align}\label{OriginProblem}
	\begin{split}
	-\epsilon\curl\Delta \curl \bm u+\alpha\curl \curl \bm u +\beta \bm u&=\bm f \quad\text{in}\ \Omega,\\
	\div \bm u&=0\quad\text{in}\ \Omega,\\
	\bm u\times \bm n&=0\quad\text{on}\ \partial\Omega,\\
	\curl \bm u&=0\quad\text{on}\ \partial\Omega.
	\end{split}
	\end{align}
	Here $\bm f\in H(\text{div}^0;\Omega):=\{\bm u\in [L^2(\Omega)]^3:\; \div \bm u=0\}$, $\alpha> 0$ and $\beta\geq 0$ are constants of moderate size, $\epsilon>0$ is a constant that can approach 0, and $\bm n$ is the unit outward normal vector to $\partial \Omega$.

Standard conforming finite element methods to solve problem \eqref{OriginProblem} require function spaces to be subspaces of $H(\grad\curl;\Omega)$, see Section \ref{notation} for its precise definition. Some $H$(gradcurl)-conforming finite elements have been constructed in the past years. So far, the construction in 2 dimensions(2D) is relatively complete \cite{WZZelement,hu2020simple,wang2021h}, however, that is not the case for 3 dimensions (3D).      
In \cite{ZhangCSIAM2021A}, one of the authors and Z. Zhang developed a tetrahedral $\grad\curl$-conforming element, which has $315$ degrees of freedom (DOFs) on each element. To reduce the number of DOFs, they, together with their collaborators, enriched  the
shape function space on each tetrahedron with piecewise-polynomial bubbles \cite{Hu2022A}, and the resulting element has only 18 DOFs. However, it is challenging to extend the idea of enriching bubbles to cubical element. The only grad curl-conforming element on cubical meshes constructed in \cite{wang2021hh} has $144$ DOFs on each cube. 

Therefore, we resort to nonconforming elements to reduce the number of DOFs. 
On tetrahedral meshes, Zheng et al. constructed the first nonconforming finite element in \cite{Zheng2011A}, which has $20$ DOFs on each element. Recently, Huang \cite{Huang2020Nonconforming} proposed a nonconforming finite element Stokes complex which includes two $\grad\curl$-nonconforming finite elements:  one coincides with the element in \cite{Zheng2011A} and the other one has only 14 DOFs.
However, the convergence rates of the two nonconforming finite elements in \cite{Huang2020Nonconforming,Zheng2011A} will deteriorate when  $\epsilon\to0$ in problem \eqref{OriginProblem}. Hence, Huang and Zhang in \cite{huang2022robust} developed two families of $H(\grad\curl)$-nonconforming  but $H$(curl)-conforming finite elements on tetrahedral meshes, 
which can solve problem \eqref{OriginProblem} when $\epsilon$ tends to zero, see also \cite{BJzhang22}.
On cubical meshes, one of the authors and her collaborators proposed a family of nonconforming finite elements with at least $48$ DOFs in \cite{zhang2022nfecubicalmeshes}. This family is also $H$(curl)-conforming, and hence can solve problem \eqref{OriginProblem} as $\epsilon$ vanishes. To further reduce the number of DOFs, in this paper, we will construct both $H(\grad\curl)$-nonconforming and $H(\curl)$-nonconforming finite elements. Different with the tetrahedral elements, the fully nonconforming elements that we will construct still work for problem \eqref{OriginProblem} when $\epsilon\rightarrow 0$.
	
To this end, we consider the following Stokes complex:
\begin{equation}\label{stokes-complex}
\begin{tikzcd}
0 \arrow{r} & \mathbb R\arrow{r}{\subset}& H^{1}(\Omega) \arrow{r}{\grad} &H(\grad\curl;\Omega)\arrow{r}{\curl} & {[H^1(\Omega)]}^3   \arrow{r}{\div} &L^2(\Omega)  \arrow{r}{} & 0.
 \end{tikzcd}
\end{equation}
We will develop the fully $\grad\curl$-nonconforming elements by constructing a discrete nonconforming Stokes complex on cubical meshes:
	\begin{equation}\label{FE-stokes-complex}
		\begin{tikzcd}
			0\arrow{r}{}& \mathbb R\arrow{r}{\subset}& S_h^r(\mathcal T_h) \arrow{r}{\grad} & \bm V_h^{r-1}(\mathcal T_h)\arrow{r}{\curl_h} &  \bm W_h(\mathcal T_h) \arrow{r}{\div_h} &  Q_h(\mathcal T_h) \arrow{r}{} & 0,
		\end{tikzcd}
	\end{equation}
	where $S_h^r(\mathcal T_h)$, $\bm V_h^{r-1}(\mathcal T_h)$, $\bm W_h(\mathcal T_h)$, and $Q_h(\mathcal T_h)$ are conforming or nonconforming finite element spaces for $H^{1}(\Omega)$, $H(\grad\curl;\Omega)$, ${[H^1(\Omega)]}^3$, and $L^2(\Omega)$, respectively. To make the number of DOFs  as minimal as possible, we use the $r$-th order serendipity finite element space  with $r=1,2$ for $S_h^r(\mathcal T_h)$.  We choose the nonconforming Stokes pair developed in \cite{zhang2009low} for  $\bm W_h(\mathcal T_h)$ and $Q_h(\mathcal T_h)$. Then the $\grad\curl$-nonconforming finite element space $\bm V_h^{r-1}(\mathcal T_h)$ is obtained as the gradient of $S_h^r(\mathcal T_h)$ and a complementary part whose $\curl_h$ falls into $\bm W_h(\mathcal T_h)$. To be precise, the shape function space $\bm V_h^{r-1}(K)$ of $\bm V_h^{r-1}(\mathcal T_h)$ is 
	\[\bm V_h^{r-1}(K)=\grad S_h^r(K)\oplus \mathfrak{p}\bm W_h(K)\]
	where $S_h^r(K)$ and $\bm W_h(K)$ are shape function spaces of $S_h^r(\mathcal T_h)$ and $\bm W_h(\mathcal T_h)$ and $\mathfrak{p}$ is the Poincar\'e operator with the precise definition given in Sections \ref{section2}.  The DOFs of $\bm V_h^{r-1}(\mathcal T_h)$ can be obtained from the DOFs of $S_h^r(\mathcal T_h)$ and $\bm W_h(\mathcal T_h)$. 
	By taking $r=1,2$, we will get two versions of $\bm V_h^{r-1}(\mathcal T_h)$. The number of DOFs on each element is 24 for $r=1$ and 36 for $r=2$. By constructing $\bm V_h^{r-1}(\mathcal T_h)$ in this way, we can show that the complex \eqref{FE-stokes-complex} is exact on contractible domains. 

    Based on the newly proposed $\bm V_h^{r-1}(\mathcal T_h)$ and $S_h^r(\mathcal T_h)$, we develop a robust mixed finite element method for solving the quad-curl singular perturbation problem \eqref{OriginProblem}. The wellposedness of the numerical scheme  holds by proving the discrete Poincar\'e inequality, the discrete inf-sup condition, and the coercivity condition. Moreover, we prove a special property of functions in 
    $\bm V_h^{r-1}(\mathcal T_h)$:
\begin{align}\label{property}
			\int_{\partial K}\bm  q\cdot (\bm v_h-\bm I_h^{r-1}\bm v_h)\times \bm n_{\partial K}\d A=0, \qquad\text{ for all } \bm  q\in \bm P_{0}(K)\text{ and }K\in\mathcal T_h,
		\end{align}
  where $\bm I_h^{r-1}$ is defined in Section \ref{section3}.
	With the interpolation results, projection error estimates, and property \eqref{property}, we obtain an optimal convergence order $\mathcal{O}(h)$ for any fixed $\epsilon$. As $\epsilon$ approaches $0$, the right-hand side of the optimal estimate will blow up, which makes the estimate useless. Therefore, we also provide a uniform convergence order $\mathcal{O}(h^{1/2})$ for any value of $\epsilon$ in the sense of the energy norm. We note that the optimal estimate for a moderate $\epsilon$ can be obtained without using property \eqref{property}, while the uniform estimate can not. The fully nonconforming finite elements on tetrahedral meshes \cite{Huang2020Nonconforming} do not possess this property, which is the reason that those elements can not work for the quad-curl singular perturbation problem.

    The rest of the paper is organized as follows. In section 2 we list some notation that will be used throughout the paper. In section 3 we define the fully $\grad\curl$-nonconforming finite element on a cube and estimate the interpolation errors. Nonconforming and exact finite element complexes are  constructed in section 4. In section 5 we use our proposed elements to solve the quad-curl singularly perturbed problem and obtain the optimal convergence order of the energy error  for any fixed $\epsilon$.  In section 6 we provide  a uniform error estimate with respect to $\epsilon$. In section 7, 
numerical examples are shown to verify the correctness and efficiency of our method. Finally, some concluding remarks are given in section 8.

	\section{Notation}\label{notation}
	We assume that $\Omega\subset\mathbb{R}^3$ is a contractible Lipschitz domain throughout the paper.  We adopt conventional notation for Sobolev spaces such as $W^{s,p}(D)$ or $W_0^{s,p}(D)$ on a sub-domain $D\subset\Omega$ furnished with the norm $\left\|\cdot\right\|_{W^{s,p}(D)}$ and the semi-norm $\left|\cdot\right|_{W^{s,p}(D)}$. In the case of $p=2$, we use notation $H^s(D)$ or $H_0^s(D)$ for spaces $W^{s,p}(D)$ or $W_0^{s,p}(D)$. The norm and semi-norm are $\|\cdot\|_{s,D}$ and $|\cdot|_{s,D}$, respectively. In the case of $s=0$, the space $H^{0}(D)$ coincides with $L^2(D)$ which is equipped with the inner product $(\cdot,\cdot)_D$ and the norm $\left\|\cdot\right\|_D$. 
	When $D=\Omega$, we drop the subscript $D$. 
	
	
	In addition to the standard Sobolev spaces, we also define
	\begin{align*}
		& H(\text{curl};D):=\{\bm u \in [L^2(D)]^3:\; \curl \bm u \in [L^2(D)]^3\},\\
		&H(\text{div};D):=\{\bm u \in [L^2(D)]^3:\; \div\bm u \in  L^2(D)\},\\
		& H(\grad\curl;D):=\{\bm u \in [L^2(D)]^3:\; \curl  \bm u \in [H^1(D)]^3\},\\
		&H^s(\curl;D):=\{\bm u \in [H^s(D)]^3:\; \curl\bm u \in [H^s(D)]^3\}.
	\end{align*}
	

	
	For a subdomain $D$, we use $P_r(D)$, or simply $P_{r}$ when there is no possible confusion, to denote the space of polynomials with degree at most $r$ on $D$. We denote $Q_{i,j,k}$ the space of polynomials in three variables $(x,y,z)$ whose maximum degrees are $i$ in $x$, $j$ in $y$, and $k$ in $z$, respectively.
	
	For a scalar function space $S$, we use $\bm S$ or $[S]^3$ to denote the $S\otimes\mathbb R^3$.

	

	Let \,$\mathcal{T}_h\,$ be a partition of the domain $\Omega$
	consisting of shape-regular cuboids. For $K=(x_1^c-h_1,x_1^c+h_1)\times (x_2^c-h_2,x_2^c+h_2)\times (x_3^c-h_3,x_3^c+h_3)$, we denote $h_K=\sqrt{h_1^2+h_2^2+h_3^2}$ as the diameter of $K$ and $h=\max_{K\in\mathcal T_h}h_K$ as the mesh size of $\mathcal {T}_h$. Denote by $\mathcal V_h$, $\mathcal E_h$, and $\mathcal F_h$ the sets of vertices, edges, and faces in the partition, and $\mathcal V_h^{\text{int}}$, $\mathcal E_h^{\text{int}}$, and $\mathcal F_h^{\text{int}}$ the sets of interior vertices, edges, and faces. 
	We also denote by $\mathcal V_h(K)$, $\mathcal E_h(K)$, and $\mathcal F_h(K)$ the sets of vertices, edges, and faces related to the element $K\in\mathcal T_h$. We use $\bm\tau_e$ and $\bm n_f$ to denote the unit tangential vector and the unit normal vector to $e\in \mathcal E_h$ and $f\in \mathcal F_h$, respectively. 
	Suppose $f=K_1\cap K_2\in \mathcal F_h$. For a function $v$ defined on $K_1\cup K_2$, we define
	$[\![v]\!]_f=f|_{K_1}-f|_{K_2}$ to be the jump across $f$. When $f\subset \partial \Omega$, $[\![v]\!]_f=v$.
	
	We use $C$ to denote a generic positive constant that is independent of $h$.
	\section{The H(gradcurl)-Nonconforming Finite Elements}\label{section2}
	In this section, we will construct $H(\grad\curl)$-nonconforming finite elements on cubical meshes. To this end, recall the  Poincar{\'e} operator $\mathfrak{p}:[C^{\infty}]^3\to [C^{\infty}]^3$ \cite{christiansen2018generalized} defined by
	\begin{equation*}
		\mathfrak{p} \bm{v}:=-\bm{x}\times\int_{0}^{1}t\bm{v}(t\bm{x}) \d t.
	\end{equation*}
	The Poincar{\'e} operator satisfies
	\begin{align}
	   & \curl \mathfrak{p} \bm{v} + \mathfrak{p}^3\div\bm v= \bm{v},\label{identity-poincare}\\
     & \grad \mathfrak{p}^1 \bm{v}+ \mathfrak{p}\curl\bm v =\bm v,\label{identity-poincare-2}
	\end{align}
	where $\mathfrak{p}^3 v =\bm x\int_0^1t^2v(t\bm x)\d t$ and $\mathfrak{p}^1 \bm v =\int_0^1\bm v(t\bm x)\cdot\bm x\d t$. It holds
 \begin{align}
    \mathfrak{p}^1\mathfrak{p} = \mathfrak{p}\mathfrak{p}^3  =0.
 \end{align}
	
For each element $K\in \mathcal{T}_h$, we define the following shape function space:
	\begin{align}\label{shapefunction}
		\bm{V}^{r-1}_h(K) :=\grad S_h^r(K)\oplus \mathfrak{p}\bm W_h(K)
	\end{align}
	where 		
	\begin{align}\label{WK}
	    \bm W_h(K)= [P_1(K)]^3\oplus \Span\left\{\begin{pmatrix}y^2\\0\\0\end{pmatrix},\begin{pmatrix}z^2\\0\\0\end{pmatrix},\begin{pmatrix}0\\x^2\\0\end{pmatrix},\begin{pmatrix}0\\z^2\\0\end{pmatrix},\begin{pmatrix}0\\0\\x^2\end{pmatrix},\begin{pmatrix}0\\0\\y^2\end{pmatrix}\right\},
	\end{align} 
	and for $r=1,2$, the shape function space $S_h^r(K)$ contains all polynomials of superlinear degree at most $r$. The superlinear degree of a polynomial is the degree with ignoring variables that enter linearly. For example, $x^2yz^3$ is a polynomial with superlinear degree $5$. To be specific, the space $S_h^2(K)$ is spanned by the monomials in $P_2(K)$ and \[xy^2,xz^2,yx^2,yz^2,zx^2,zy^2,xyz,xyz^2,x^2yz,xy^2z.\]
	We have $\dim S_h^1(K)=8$ and $\dim S_h^2(K)=20$. 
	The space $\bm W_h(K)$ is the same as the space $\bm V_T^{(3)}$ defined in \cite{zhang2009low} and $\dim \bm W_h(K) = 18$.
	
The right-hand side of \eqref{shapefunction} is a direct sum. In fact, if $\bm v\in \grad S_h^r(K)\cap\mathfrak{p}\bm W_h(K)$, then $\curl\bm v=0$ and $\mathfrak p^1\bm v =0$. According to \eqref{identity-poincare-2}, we have $\bm v=0$.
From the definition of $\bm{V}^{r-1}_h(K)$, we have $\bm P_{r-1}(K)\subset \bm{V}^{r-1}_h(K)$ and 
	\begin{align*}
		\text{dim}\ \bm{V}^{r-1}_h(K)=\dim \grad S_h^r(K)+\dim \mathfrak p\bm W_h(K) = \left\{\begin{array}{l}
			24,\quad r=1,\\
			36,\quad r=2.
		\end{array}
		\right.
	\end{align*}
We define the following DOFs for $\bm v\in \bm{V}^{r-1}_h(K)$:
	\begin{align}
		M_e(\bm{v})&=\int_e \bm{v}\cdot\bm{\tau}_e q\d s,\ \forall q\in P_{r-1}(e)\ \text{ at all edges }e\in \mathcal E_h(K),\label{Dof2}\\
        M_f(\bm{v})&=\int_f \curl\bm{v}\times \bm{n}_f \d A\  \text{ at all faces } f\in \mathcal F_h(K).\label{Dof1}
	\end{align}
	\begin{remark}\label{remark_freedom}
		The DOFs $M_f(\bm{v})$ can be  equivalently defined as
		\begin{align}
			&M_f^1(\bm{v})=\int_f \curl\bm{v}\cdot\bm{\tau}_f^{1} \d A\  \text{ at all faces } f\in \mathcal F_h(K),\label{Dof1_1}\\
			&M_f^2(\bm{v})=\int_f \curl\bm{v}\cdot\bm{\tau}_f^{2} \d A\  \text{ at all faces } f\in \mathcal F_h(K),\label{Dof1b}
		\end{align}
		where $\bm{\tau}_f^1,\ \bm{\tau}_f^2$ are two unit vectors parallel to the two non-colinear edges in face $f$.
	\end{remark}
	\begin{lemma}\label{boundednessofDOF}
The DOFs $M_e(\bm v)$ and $M_f(\bm v)$ are well-defined for $\bm v\in  H^1(\curl;K)$.
	\end{lemma}
\begin{proof}
The boundedness of the DOFs $M_f(\bm v)$ is trivial. To prove the boundedness of DOFs $M_e(\bm v)$, 
we apply a similar idea to the proof of \cite[Lemma 5.38]{Monk2003}. For an edge $e$, $f$ is the face containing the edge on its boundary. Given a polynomial $q\in P_{r-1}(e)$, we extend it by 0 to a function on $\partial f$ (still denoted by $q$). Such a function $q$ is in $W^{1-1/p^\prime,p^\prime}(\partial f)$ with $p'=p/(p-1)>1$ and  $p>2$. 
Suppose $K$ is an element containing the face $f$ on its boundary, we have
		\begin{align*}
		&\int_e\bm u\cdot \bm \tau q  \d s=\int_f \text{rot}_f \bm u_f q\d A-\int_f\bm u\cdot \curl_f q\d A\\
			&\le\|\curl \bm u\|_{L^2(f)}\|q\|_{L^2(f)}+\|\bm u\|_{L^p(f)}\|\curl_f q\|_{L^{p'}(f)}\\
			&\le\|\curl \bm u\|_{L^2(f)}\|q\|_{W^{1,p'}(f)}+\|\bm u\|_{L^p(f)}\|q\|_{W^{1,p'}(f)}\  (\text{embedding theorem \cite[Theorem 4.57]{demengel2012functional}})\\
			&\le \|\curl \bm u\|_{L^2(f)}\|q\|_{W^{1-1/p',p'}(e)}+\|\bm u\|_{L^p(f)}\|q\|_{W^{1-1/p',p'}(e)}\  \text{(\cite[Theorem 1.5.1.2]{grisvard2011elliptic})}\\
			&\le C(\| \curl\bm u\|_{L^2(f)}+\|\bm u\|_{H^{1/2}(f)}) \ (\text{embedding theorem \cite[Theorem 4.57]{demengel2012functional}})\\
            &\le C(\| \curl\bm u\|_{1,K}+\|\bm u\|_{1,K})\ (\text{trace theorem \cite[Theorem 1.5.1.2]{grisvard2011elliptic}}).
		\end{align*} 
\end{proof}
	\begin{lemma}
		The DOFs \eqref{Dof1}-\eqref{Dof2} are unisolvent for the shape function space $\bm{V}^{r-1}_h(K)$.
	\end{lemma}
	\begin{proof}
		First of all, the number of DOFs \eqref{Dof1}-\eqref{Dof2} is $6\times2+12r$ which is same as the dimension of $\bm{V}^{r-1}_h(K)$. Then it suffices to show that if all the DOFs \eqref{Dof1}-\eqref{Dof2} vanish on a function $\bm{v}=\grad q+ \mathfrak p \bm w\in \bm{V}^{r-1}_h(K)$ with $q\in S_h^r(K)$ and $\bm w \in \bm W_h(K)$, then $\bm{v}=0$.
		
		For each $f\in\mathcal{F}_h(K)$, by the integration by parts on face $f$, we have
		\begin{align}
			\int_f \curl \bm{v}\cdot\bm{n}_f \d A=\int_{\partial f}\bm{v}\cdot \bm{\tau}\d s=0,
		\end{align}
	 which together with vanishing DOFs in \eqref{Dof1}, we obtain
		\begin{align}
			\int_f \curl \bm{v}\d A=\bm{0}.
		\end{align}
		By \eqref{identity-poincare} and the definition of $\bm{V}^{r-1}_h(K)$, we have $\curl\bm{v}\in \bm{W}_h(K)$. 
		The unisolvence of the DOFs for $\bm{W}_h(K)$ \cite{zhang2009low} leads to $\text{curl }\bm{v}=\bm{0}$, which implies $\curl\mathfrak p \bm w=\bm w-{\mathfrak p}^3\div\bm w=0$. Then $\bm{v}=\grad q+\mathfrak p \bm w = \grad q+\mathfrak p {\mathfrak p}^3\div\bm w=\grad q$. Using the DOFs \eqref{Dof2}, we can obtain $\partial_{\bm{\tau}_e}q=0$, then we can choose $q\in S^r_h(K)$ satisfying $q|_e=0$ for each $e\in\mathcal{E}(K)$. Since the superlinear degree of $q$ is no more than $2$, we can get $\bm{v}=0$.
	\end{proof}
	
	
For the reference element $\hat K=(-1,1)^3$, we denote the DOFs for $\hat{\bm v}\in\bm{V}^{r-1}_h(\hat K)$ by
\[M_{\hat{f}}(\hat{\bm{v}})\cup M_{\hat{e}}(\hat{\bm{v}}).\]
We relate the shape function $\bm v$ on $K$ and  $\hat{\bm v}$ on $\hat K$ by  
	\begin{align}\label{tran}
		\bm{v}\circ F_K =B_K^{-T}\hat{\bm{v}},
	\end{align}
	where $F_K(\bm{x}):=B_K\hat{\bm{x}}+\bm{b}_K$ is a affine mapping from reference element $\hat{K}$ to an element $K\in\mathcal T_h$. For technical reasons, we assume the elements in $\mathcal T_h$ have all edges parallel to the coordinates axes, then $B_K$ is a diagonal matrix.


	 	For any $K\in\mathcal{T}_h$, we  define an interpolation operator $\bm{R}^{r-1}_K:H^1(\curl;K)\rightarrow \bm V_h^{r-1}(K)$ by
	\begin{align}
		M_e(\bm{v}-\bm{R}^{r-1}_K\bm{v})=0, M^1_f(\bm{v}-\bm{R}^{r-1}_K\bm{v}), \text{ and } M^2_f(\bm{v}-\bm{R}^{r-1}_K\bm{v})=0.
	\end{align}
	Similarly, we can define $\bm{R}^{r-1}_{\hat K}$.
	The following lemma establishes the relationship between the interpolation on $K$ and the interpolation on $\hat{K}$.  

	\begin{lemma}\label{hatPi}
		Suppose that $\bm{R}^{r-1}_K \bm{v}$ is well defined. Then under the transformation \eqref{tran}, we have
		\begin{align*}
			\widehat{\bm{R}^{r-1}_K \bm{v}}=\bm{R}^{r-1}_{\hat K}\hat{\bm{v}}.
		\end{align*}
	\end{lemma}
	\begin{proof}
	The interpolations $\widehat{\bm{R}^{r-1}_K \bm{v}}$ and $\bm{R}^{r-1}_{\hat K}\hat{\bm{v}}$ are defined by DOFs $M_e(\bm v)\cup M_f^1(\bm v)\cup M_f^2(\bm v)$ and $M_{\hat e}(\hat{\bm v})\cup M_{\hat f}^1(\hat{\bm v})\cup M_{\hat f}^2(\hat{\bm v})$, respectively. Under the transformation \eqref{tran} , each DOF in $M_e(\bm v)\cup M_f^1(\bm v)\cup M_f^2(\bm v)$ is a constant multiple of the corresponding DOF in  $M_{\hat e}(\hat{\bm v})\cup M_{\hat f}^1(\hat{\bm v})\cup M_{\hat f}^2(\hat{\bm v})$. For example
\begin{align*}
			&\quad\int_f\curl\bm{u}\cdot \bm{\tau}_f^1 \d A=\int_{\hat{f}}\frac{1}{\det(B_K)}B_K\widehat{\curl}\hat{\bm{u}}\cdot \hat{\bm{\tau}}_f^1 \frac{\text{area}(f)}{\text{area}(\hat{f})}\d \hat{A}\\
			&=\frac{\|B_K\hat{\bm{\tau}}_f^1\|}{\|B_K\hat{\bm n}_f\|}\int_{\hat{f}}\widehat{\curl}\hat{\bm{u}}\cdot \hat{\bm{\tau}}_f^1 \d \hat{A},
		\end{align*}
		where we have used the facts that $B_K$ is a diagonal matrix and $\hat{\bm{\tau}}_f^1={\bm{\tau}}_f^1$.
		Therefore the two interpolations are identical according to \cite[Proposition 3.4.7]{brenner2008mathematical}. 
	\end{proof}
With the help of Lemma \ref{hatPi}, we can show that the interpolation operator has the following approximation property.
	
	\begin{theorem}\label{Interp_err_R}
		Suppose that $\bm{u}\in \bm H^{s}(K)$ and $\curl\bm{u}\in \bm H^{l}(K)$ with $l\geq s\geq 1$, then the following error estimates hold.
		\begin{align}
			\|\bm{u}-\bm{R}^{r-1}_K\bm{u}\|_K&\leq Ch^{\min\{s,r\}}\big(|\bm{u}|_{s,K}+|\curl\bm{u}|_{s,K}\big),\label{Interp_err_R0}\\
			\|\curl(\bm{u}-\bm{R}^{r-1}_K\bm{u})\|_{j,K}&\leq Ch^{\min\{l,2\}-j}|\curl\bm{u}|_{l,K}, \ j=0,1.\label{Interp_err_R1}
		\end{align}
	\begin{proof}
	    By following the proof of \cite[Theorem 5.41]{Monk2003}, we can prove \eqref{Interp_err_R0} and \eqref{Interp_err_R1} by using Lemmas \ref{boundednessofDOF} and \ref{hatPi}, the fact that $\bm P_{r-1}(K)\subset \bm V_h^{r-1}(K)$, and the approximation property of $\bm W_h(K)$ \cite[(3.3)]{zhang2009low}. 
	\end{proof}
	\end{theorem}

	\section{Nonconforming finite element Stokes complexes}\label{section3}
	\subsection{A finite element de Rham complex on cubical meshes}\label{derhamFEM}
	In this section, we present the precise definition of each space involved in the complex \eqref{FE-stokes-complex}.

	\paragraph{{\bf Space $S_h^r(\mathcal T_h)$}}This is the $H^1$-conforming serendipity finite element space \cite{arnold2011serendipity}. 
	The shape function space is $S_h^r(K)$ and the DOFs for $u\in S_h^r(\mathcal T_h)$ are
	\begin{itemize}
		\item function values $u(v)$ at all vertices $v\in\mathcal V_h$,
		\item moments $\int_e uq\d s,\ \forall q\in P_{r-2}(e)$ at all edges $e\in \mathcal E_h$.
	\end{itemize}
	The above DOFs lead to the canonical interpolation $\pi_h^r$: $H^{\frac{3}{2}+\delta}(\Omega)\rightarrow S^r_h(\mathcal T_h)$ with $\delta>0$. 
	\paragraph{{\bf Space $\bm V^{r-1}_h(\mathcal T_h)$}}The space is defined by
	\begin{align*}
	    \bm V^{r-1}_h(\mathcal T_h)=\{\bm u\in\bm L^2(\Omega): &\ \bm u|_{K}\in \bm V^{r-1}_h(K)\text{ for all }K\in\mathcal T_h  \text{ and }\bm u\text{ is uni-valued}\\&\text{ across elements on the DOFs in } M_e(\bm u)\cup M_f(\bm u)\}.
	\end{align*}
		We define a global interpolation operator $\bm{R}^{r-1}_h:\bm H^1(\curl;\Omega)\rightarrow \bm{V}_h^{r-1}(\mathcal T_h)$ 
	element-wisely by
	\begin{align*}
		(\bm{R}^{r-1}_h\bm{u})|_K=\bm{R}^{r-1}_K(\bm{u}|_K)\quad\text{for any }K\in\mathcal{T}_h.
	\end{align*}
	\paragraph{{\bf Space $\bm W_h(\mathcal T_h)$}}For $\bm W_h(\mathcal T_h)$, we use the $\bm H^1$-nonconforming finite element space in \cite{zhang2009low}, which was constructed for the Darcy-Stokes-Brinkman problem. 
 	The shape function space is $\bm W_h(K)$ defined by \eqref{WK}.
	
	The DOFs for $\bm u\in \bm W_h(\mathcal T_h)$ are
	\begin{itemize}
		\item moments $\int_f \bm u\d A$ at all faces $f\in \mathcal F_h$.
	\end{itemize}
	Applying the above DOFs, we can define an interpolation $\bm \Pi_h: \bm H^1(\Omega)\rightarrow\bm W_h(\mathcal T_h)$, whose restriction on $K$ is denoted as $\bm \Pi_K$.
	
	\paragraph{{\bf Space $Q_h(\mathcal T_h)$}}
	The shape function space of $Q_h(\mathcal T_h)$ on $K$ is $P_{0}(K)$. For $ u\in Q_h(\mathcal T_h)$, the DOFs are
	\begin{itemize}
		\item moments $\int_K u\d V$ at all elements $K\in \mathcal T_h$.
	\end{itemize}
	The above DOFs lead to the canonical interpolation $\mathcal P_h$: $L^2(\Omega)\rightarrow Q_h(\mathcal T_h)$, whose restriction on $K$ is denoted as $\mathcal P_K$.

 We summarize the interpolations defined by the DOFs in the following diagram.
	\begin{equation}\label{FE-interp-complex-digram}
		\begin{tikzcd}
			0\arrow{r}{}& \mathbb R\arrow{r}{\subset}& H^{2}(\Omega)\arrow{d}{\pi_h^r} \arrow{r}{\grad} & H^1(\curl;\Omega)\arrow{d}{\bm R_h^{r-1}}\arrow{r}{\curl} &  \bm H^1(\Omega) \arrow{d}{\bm \Pi_h}\arrow{r}{\div} &  L^2(\Omega) \arrow{d}{\mathcal P_h}\arrow{r}{} & 0\\
			0\arrow{r}{}& \mathbb R\arrow{r}{\subset}& S_h^r(\mathcal T_h) \arrow{r}{\grad} & \bm V^{r-1}_h(\mathcal T_h)\arrow{r}{\curl_h} &  \bm W_h(\mathcal T_h) \arrow{r}{\div_h} &  Q_h(\mathcal T_h) \arrow{r}{} & 0
		\end{tikzcd}
	\end{equation}
	\begin{lemma}\label{commuting}
		The interpolations in the diagram \eqref{FE-interp-complex-digram} commute with the differential operators, i.e.,
		\begin{align*}
			\bm R_h^{r-1}\grad  = \grad \pi_h^r, \ \	\bm \Pi_h\curl = \curl_h\bm R_h^{r-1}, \ \ \text{and}   \  \ \mathcal P_h\div=\div_h\bm\Pi_h.
		\end{align*}
	\end{lemma}
	\begin{proof}
		It is straightforward to prove the lemma by following the argument in the proof of \cite[Lemma 4.5]{hu2020simple}.
	\end{proof}
	\begin{lemma}\label{exactness1}
		The complex \eqref{FE-stokes-complex} is exact when $\Omega$ is contractible. 
	\end{lemma}
	\begin{proof}
		To prove the complex is exact, we shall show that the complex is exact at each space. We start with $\bm V_h^{r-1}(\mathcal T_h)$. We suppose $\bm u_h\in \bm V_h^{r-1}(\mathcal T_h)$ and $\curl_h\bm u_h=0$, and we show that there exists a $p_h\in S_h^r(\mathcal T_h)$ such that $\bm u_h=\grad p_h$. Since $\curl_h\bm u_h=0$, we can write $\bm u_h|_K=\grad p_K$ with $p_K\in S_h^r(K)$ for any $K\in \mathcal T_h$. Suppose $f\in \mathcal F_h$ is a common face of $K_1,K_2\in \mathcal T_h$. Restricted on $f$, the polynomials $p_{K_1}$ and $p_{K_2}$ are in $S_h^{r}(f)$ (the counterpart of $S_h^r(K)$ in 2D). The DOFs $\int_e\bm u_h\cdot\bm \tau_e q \d s$ for any $q\in P_{r-1}(e)$ and $e\in\mathcal E_h(f)$ yield $\grad p_{K_1}= \grad p_{K_2}$ on each $e\in\mathcal E_h(f)$, and hence $p_{K_1}-p_{K_2}=c_i$ for some $c_i$ on $e_i$, $i=1,2,3,4$. Restricting to four vertices of $f$ leads to $c_1=c_2=c_3=c_4=c$, and hence $p_{K_1}=p_{K_2}+c$ on $f$. We then pick a $c$ such that $p_{K_1}=p_{K_2}$. Therefore, we can glue together all $p_K$ for $K\in\mathcal T_h$ to get a function $p_h\in S_h^r(\mathcal T_h)$. 
		The exactness at  $Q_h(\mathcal T_h)$ follows from the exactness at the same position of smooth Stokes complex and the commutativity of interpolations and $\div_h$ (Lemma \ref{commuting})\cite{zhang2022nfecubicalmeshes}. 
		
		Finally, we prove the exactness at $\bm W_h(\mathcal T_h)$ by a dimension count. Let $\mathcal K_h$ be the number of elements in $\mathcal T_h$. Then
		\begin{align*}
			&\dim S_h^r(\mathcal T_h) = \mathcal V_h+(r-1)\mathcal E_h,
			&\dim \bm V_h^{r-1}(\mathcal T_h) = r\mathcal E_h+2\mathcal F_h,\\
			&\dim\bm W_h(\mathcal T_h) = 3\mathcal F_h,
			&\dim Q_h(\mathcal T_h) = \mathcal K_h,
		\end{align*}
		which together with Euler's formula leads to
		\begin{align*}
			\dim S_h^r(\mathcal T_h)-\dim \bm{V}^{r-1}_h(\mathcal T_h)+\dim \bm W_h(\mathcal T_h)-\dim  Q_h(\mathcal T_h)=\mathcal V_h-\mathcal E_h+\mathcal F_h-\mathcal K_h=0.
		\end{align*}
	\end{proof}

	Define
	\begin{align*}
		\mathring S_h^r(\mathcal T_h)&=S_h^r(\mathcal T_h)\cap H_0^1(\Omega),\ \ 
		\mathring{Q}_h(\mathcal T_h)= Q_h(\mathcal T_h)\cap L_0^2(\Omega),\\
		\mathring{\bm V}^{r-1}_h(\mathcal T_h)&=\{\bm u\in {\bm V}_h^{r-1}(\mathcal T_h): \text{the DOFs \eqref{Dof2} and \eqref{Dof1} vanish on }\partial\Omega\},\\
		\mathring{\bm W}_h(\mathcal T_h)&=\{\bm u\in {\bm W}_h(\mathcal T_h): \text{the DOFs $\int_f \bm u\d A$ vanish for  }f\subset \partial\Omega\}.
	\end{align*}
	Proceeding as the proof of Lemma \ref{exactness1}, we can prove the following exactness of the complex with vanishing boundary DOFs. 
	\begin{lemma}\label{exactness}
		The complex
		\begin{equation}\label{FE-stokes-complex2}
			\begin{tikzcd}
				0\arrow{r}{\subset}&\mathring S^r_h(\mathcal T_h)(\mathcal T_h) \arrow{r}{\grad} &\mathring{ \bm V}_h^{r-1}(\mathcal T_h)\arrow{r}{\curl_h} &  \mathring{\bm W}_h(\mathcal T_h) \arrow{r}{\div_h} &  \mathring{Q}_h(\mathcal T_h) \arrow{r}{} & 0
			\end{tikzcd}
		\end{equation}
		is exact when $\Omega$ is contractible. 
	\end{lemma}

	\section{Application of the Nonconforming Elements to the Model Problem}
	
	We define $H_0(\grad\curl;\Omega)$ with vanishing boundary conditions:
	\begin{align*}
		H_0(\grad\curl;\Omega):=&\left\{\bm u \in H(\grad\curl;\Omega): {\bm n}\times\bm u=0\; \text{and}\ \curl \bm u=0\; \text{on}\ \partial \Omega\right\}.
	\end{align*}
	
	To deal with the divergence-free condition, we introduce a Lagrange multiplier $p$. The mixed variational formulation is to find $(\bm u;p)\in H_0(\grad\curl;\Omega)\times H^1_0(\Omega)$  such that
	\begin{equation}\label{prob2}
		\begin{split}
			a(\bm u,\bm v)+b(\bm v,p)&=(\bm f, \bm v),\quad \forall \bm v\in H_0(\grad\curl;\Omega),\\
			b(\bm u,q)&=0,\quad \forall q\in H_0^1(\Omega),
		\end{split}
	\end{equation}
	with $a(\bm u,\bm v):=\epsilon(\grad\curl\bm u, \grad\curl\bm v) +\alpha( \curl \bm u,\curl \bm v)+\beta (\bm u,\bm v)$ and $b(\bm v,p)=(\bm v,\nabla p)$.
	
	Taking $\bm v=\grad p$ in \eqref{prob2} and using the vanishing boundary condition of $p$, we can obtain that $p\equiv 0$.
	
	The nonconforming finite element method for \eqref{prob2} seeks $(\bm u_h;p_h)\in \mathring{\bm V}^{r-1}_h(\mathcal T_h) \times\mathring S^r_h(\mathcal T_h)$  such that
	\begin{equation}\label{prob22}
		\begin{split}
			a_h(\bm u_h,\bm v_h)+b(\bm v_h,p_h)&=(\bm f, \bm v_h),\quad \forall \bm v_h\in \mathring{\bm V}^{r-1}_h(\mathcal T_h),\\
			b(\bm u_h,q_h)&=0,\quad \forall q_h\in \mathring S^r_h(\mathcal T_h),
		\end{split}
	\end{equation}
	where $a_h(\bm u_h,\bm v_h):=\sum_{K\in\mathcal T_h}\big(\epsilon(\grad\curl\bm u_h, \grad\curl\bm v_h)_K +\alpha(\curl\bm u_h, \curl\bm v_h)_K + \beta(\bm u_h,\bm v_h)_K\big)$.
	Define 
	\[X_h=\{\bm v\in \mathring{\bm V}^{r-1}_h(\mathcal T_h):b(\bm v,q)=0, \ \forall q\in \mathring S^r_h(\mathcal T_h)\},\]
	and
	\[\|\bm v_h\|_{a_h}^2=a_h(\bm v_h,\bm v_h),\ \ \|\bm v_h\|_{\grad\curl,h}^2=\epsilon|\curl_h\bm v_h|^2_{1,h}+\|\curl_h\bm v_h\|^2+\|\bm v_h\|^2,\]
	where $|\bm v|^2_{1,h}:=\sum_{K\in\mathcal T_h}\|\grad \bm v\|_K^2$ and $\curl_h$ means taking $\curl$ element by element.
	To show the well-posedness of numerical scheme \eqref{prob22}, we will first prove  the discrete Poincar\'e inequality for functions in $X_h$. To this end, we follow the idea in \cite{Huang2020Nonconforming} and define the interpolation
	$\bm I_K^{r-1}:\bm H^2(K)\longrightarrow \grad S_h^r(K)\oplus \mathfrak p [Q_{1,0,0}\times Q_{0,1,0}\times Q_{0,0,1}]$ for $K\in\mathcal{T}_h$ by
	\begin{align}\label{Hcurlinter}
		\int_e\bm I_K^{r-1} \bm v\cdot \bm \tau_e q\d s=\int_e  \bm v\cdot \bm \tau_e q\d s\ \text{for all } q\in P_{r-1}(e) \text{ and }e\in \mathcal{E}_h(K).
	\end{align}	
	Then we have \cite[Theoreoms 6.7, 6.3]{Monk2003}
	\begin{align}\label{curlestimate}
		\|\curl (\bm v-\bm I_K^{r-1}\bm v)\|_{K}\le Ch_K|\curl \bm v|_{1,K},\ \  \forall\ \bm v\in \bm H^2(K).
	\end{align}
	Proceeding as in the proof of \cite[Lemma 4.2]{Huang2020Nonconforming}, it holds that
	\begin{align}\label{L2estimate}
		\|\bm v-\bm I_K^{r-1}\bm v\|_{K}\le Ch_K\|\curl \bm v\|_{K},\ \ \forall\ \bm v\in \bm V_h^{r-1}(K).
	\end{align}
	Define
	\begin{align*}
	    &\bm H_{t,r-1}^2(\mathcal T_h)=\{\bm u\in \bm L^2(\Omega): \bm u|_K\in \bm H^2(K)\ \text{for each} \ K\in \mathcal{T}_h\text{ and }\nonumber \\
	    &\qquad\qquad  \ \ \int_e\bm u\cdot\bm\tau_eq\d s\text{ for $q\in P_{r-1}(e)$ is single-valued across elements for any } e\in \mathcal E_h\},\\ 
	    &\bm V_{h,c}^{r-1} = \{\bm u\in H(\curl;\Omega):\bm u|_K\in \grad S_h^r(K)\oplus \mathfrak p [Q_{1,0,0}\times Q_{0,1,0}\times Q_{0,0,1}]\ \text{for each} \ K\in \mathcal{T}_h\}.
	\end{align*}
	Then $\bm I_h^{r-1}: \bm H_{t,r-1}^2(\mathcal T_h)\longrightarrow\bm V_{h,c}^{r-1}$ is determined by 
	\[(\bm I_h^{r-1}\bm v)|_K:=\bm I_K^{r-1}(\bm v|_K)\text{ for all }K\in \mathcal{T}_h.\]
	
	For $\bm v\in {\bm V}^{r-1}_h(\mathcal T_h)$, since $\int_e\bm v\cdot\bm\tau_eq\d s$ is single-valued for $q\in P_{r-1}(e)$ and $e\in \mathcal E_h$, ${\bm V}^{r-1}_h(\mathcal T_h)\subset \bm H_{t,r-1}^2(\mathcal T_h)$, and hence $\bm I_h^{r-1}\bm v$ is well-defined. 
	
	For simplicity, we denote $\bm I_h^{r-1}$ as $\bm I_h$ when $r=1$.

	Now we are ready to prove the discrete Poincar\'e inequality. 
	\begin{lemma}
		The following discrete Poincar\'e inequality holds 
		\begin{align}\label{poincare}
			\|\bm v_0\|\le C_P\|\curl_h \bm v_0\|, \qquad\forall \bm v_0\in X_h.
		\end{align}
	\end{lemma}
	\begin{proof}
		According to the discrete Helmholtz decomposition \cite{Monk2003}, it follows that
		\begin{align}\label{helm-1}
			\bm I_h\bm v_0=\grad q_h+\bm r_h,
		\end{align}
		with  $q_h\in \mathring S^1_h(\mathcal T_h)$ and $\bm r_h\in Y_h=\{\bm v_h\in \bm V_{h,c}^{r-1}\cap H_0(\curl;\Omega): b(\bm v_h, q)=0,\ \forall q\in \mathring S^1_h(\mathcal T_h)\}$. 
	By \eqref{helm-1} and the discrete Poincar\'e inequality for $Y_h$, we can arrive at
		\begin{align*}
			\|\bm I_h\bm v_0\|^2&=(\bm I_h\bm v_0, \grad q_h+\bm r_h)\\
			&=(\bm I_h\bm v_0-\bm v_0, \grad q_h)+(\bm I_h\bm v_0, \bm r_h)\\
			&\le \|\bm I_h\bm v_0-\bm v_0\|\|\bm I_h\bm v_0\|+C\|\bm I_h\bm v_0\|\|\curl\bm r_h\|\\
			&= \|\bm I_h\bm v_0-\bm v_0\|\|\bm I_h\bm v_0\|+C\|\bm I_h\bm v_0\|\|\curl \bm I_h\bm v_0\|,
		\end{align*}
		which implies
		\begin{align*}
			\|\bm I_h\bm v_0\|\le \|\bm I_h\bm v_0-\bm v_0\|+C\|\curl\bm I_h\bm v_0\|.
		\end{align*}
		Therefore, we have
		\begin{align*}
			\|\bm v_0\|&\le \|\bm v_0-\bm I_h\bm v_0\|+\|\bm I_h\bm v_0\|\\
			&\le 2\|\bm v_0-\bm I_h\bm v_0\|+C\|\curl \bm I_h\bm v_0\|\\
			&\le 2\|\bm v_0-\bm I_h\bm v_0\|+C\|\curl\bm I_h\bm v_0-\curl_h\bm v_0\|+C\|\curl_h\bm v_0\|.
		\end{align*}
		Then \eqref{poincare} follows immediately from \eqref{curlestimate}, \eqref{L2estimate}, and the inverse inequality.
	\end{proof}
	\begin{remark}
		Different with the proof of \cite[lemma 4.3]{Huang2020Nonconforming}, our method does not require the existence of bounded commuting projection operators.
	\end{remark}
	
	We are now in a position to state the wellposedness of the numerical scheme. The finite element spaces $\mathring{\bm V}^{r-1}_h(\mathcal T_h)$ and  $\mathring S^r_h(\mathcal T_h)$ satisfy:
	\begin{itemize}
	
		\item the inf-sup condition
		\begin{align}\label{infsup2}
			\sup_{\bm w_h\in\mathring{\bm V}^{r-1}_h(\mathcal T_h)}\frac{b(\bm w_h,s_h)}{\|\bm w_h\|_{\grad\curl,h}}\geq \frac{(\grad s_h,\grad s_h)}{\|\grad s_h\|_{\grad\curl,h}}=\|\grad s_h\|\ge 
			C\|s_h\|_1, \quad\forall s_h\in \mathring S^r_h(\mathcal T_h),
		\end{align}
		\item the coercivity condition
		\begin{align}\label{coer}
			\|\bm v_h\|_{a_h}\geq c\|\bm v_h\|_{\grad\curl,h}\text{ for }\bm v_h\in X_h,
		\end{align}
		which is directly obtained from \eqref{poincare}.
	\end{itemize} 
Then the finite element scheme is well-posed.
	\begin{theorem}\label{converg2}	Problem \eqref{prob22} has a unique solution $(\bm u_h;p_h)$  such that $p_h=0$ and
		\begin{align}\label{err-abs}
			\|\bm u-\bm u_h\|_{\grad\curl,h}&\leq C\Big(\inf_{\bm v_h\in \mathring{\bm V}^{r-1}_h(\mathcal T_h)}\|\bm u-\bm v_h\|_{\grad\curl,h}+\sup_{\bm v_0\in X_h}\frac{E_h(\bm u,\bm v_0)}{\|\bm v_0\|_{a_h}}\Big),
		\end{align}
		where the consistency error 
		\begin{align*}
			E_h(\bm u,\bm v_0)=&-\sum_{f\in \mathcal{F}_h}\epsilon\langle{\partial_{\bm n_f} (\curl\bm u)}, [\![\curl_h\bm v_0]\!]\rangle_f+\sum_{f\in\mathcal F_h}\alpha\langle\curl\bm u, [\![\bm v_0\times \bm n_f]\!]\rangle_f\\
			&+\sum_{f\in \mathcal{F}_h}\epsilon\langle \Delta\curl\bm u,[\![\bm v_0\times \bm n_f]\!]\rangle_f.
		\end{align*}
	\end{theorem}
	\begin{proof}
		Define 
		\begin{align}\label{E2h}
		E_h(\bm u,\bm v_h):=a_h(\bm u,\bm v_h)+b(\bm v_h,p)-(\bm f,\bm v_h).
		\end{align}
		Then estimate \eqref{err-abs} follows from \cite[Theorem 3.3]{zhang2022nfecubicalmeshes}. From integration by parts and the fact that $p=0$, we get
		\begin{align*}
		&E_h(\bm u,\bm v_h)=	a_h(\bm u,\bm v_h)-(\bm f,\bm v_h)\\
		=&\sum_{K\in\mathcal T_h}\epsilon(\grad\curl\bm u, \grad\curl\bm v_h)_K +\sum_{K\in\mathcal T_h}\alpha(\curl\bm u, \curl\bm v_h)_K + \sum_{K\in\mathcal T_h}\beta(\bm u,\bm v_h)_K-(\bm f,\bm v_h)\\
			=&-\epsilon\sum_{K\in\mathcal T_h}(\Delta\curl\bm u, \curl\bm v_h)_K +\sum_{K\in\mathcal T_h}\alpha(\curl\curl\bm u, \bm v_h)_K + \sum_{K\in\mathcal T_h}\beta(\bm u,\bm v_h)_K-(\bm f,\bm v_h)\\
			&
			+\sum_{f\in \mathcal{F}_h}\epsilon\langle {\partial_{\bm n_f} (\curl\bm u)}, [\![\curl_h\bm v_h]\!]\rangle_f+\sum_{f\in\mathcal F_h}\alpha\langle\curl\bm u, [\![\bm v_h\times \bm n_f]\!]\rangle_f\\
			=&\sum_{f\in \mathcal{F}_h}\epsilon\langle {\partial_{\bm n_f} (\curl\bm u)}, [\![\curl_h\bm v_h]\!]\rangle_f+\sum_{f\in\mathcal F_h}\alpha\langle\curl\bm u, [\![\bm v_h\times \bm n_f]\!]\rangle_f\\
			&-\sum_{f\in \mathcal{F}_h}\epsilon\langle \Delta\curl\bm u,[\![\bm v_h\times \bm n_f]\!]\rangle_f.
		\end{align*}
	\end{proof}
	
	To estimate the consistency error $E_h$, we need the following lemma, which is the key for the element working for problem \eqref{OriginProblem} with small $\epsilon$.
	\begin{lemma}\label{jump-property}
		For any $\bm v_h\in {\bm V}^{r-1}_h(\mathcal T_h)$ and $K\in\mathcal T_h$, we have
		\begin{align}\label{jump-pro}
			\int_{\partial K}\bm  q\cdot (\bm v_h-\bm I_h^{r-1}\bm v_h)\times \bm n_{\partial K}\d A=0, \qquad\text{ for } \bm  q\in \bm P_{0}(K),
		\end{align}
  where $\bm n_{\partial K}$  is the unit outward normal vector to $\partial K$. 
	\end{lemma}
	\begin{proof} 
	\begin{figure}[htb]
			\centering
			\includegraphics[width=0.33\textwidth]{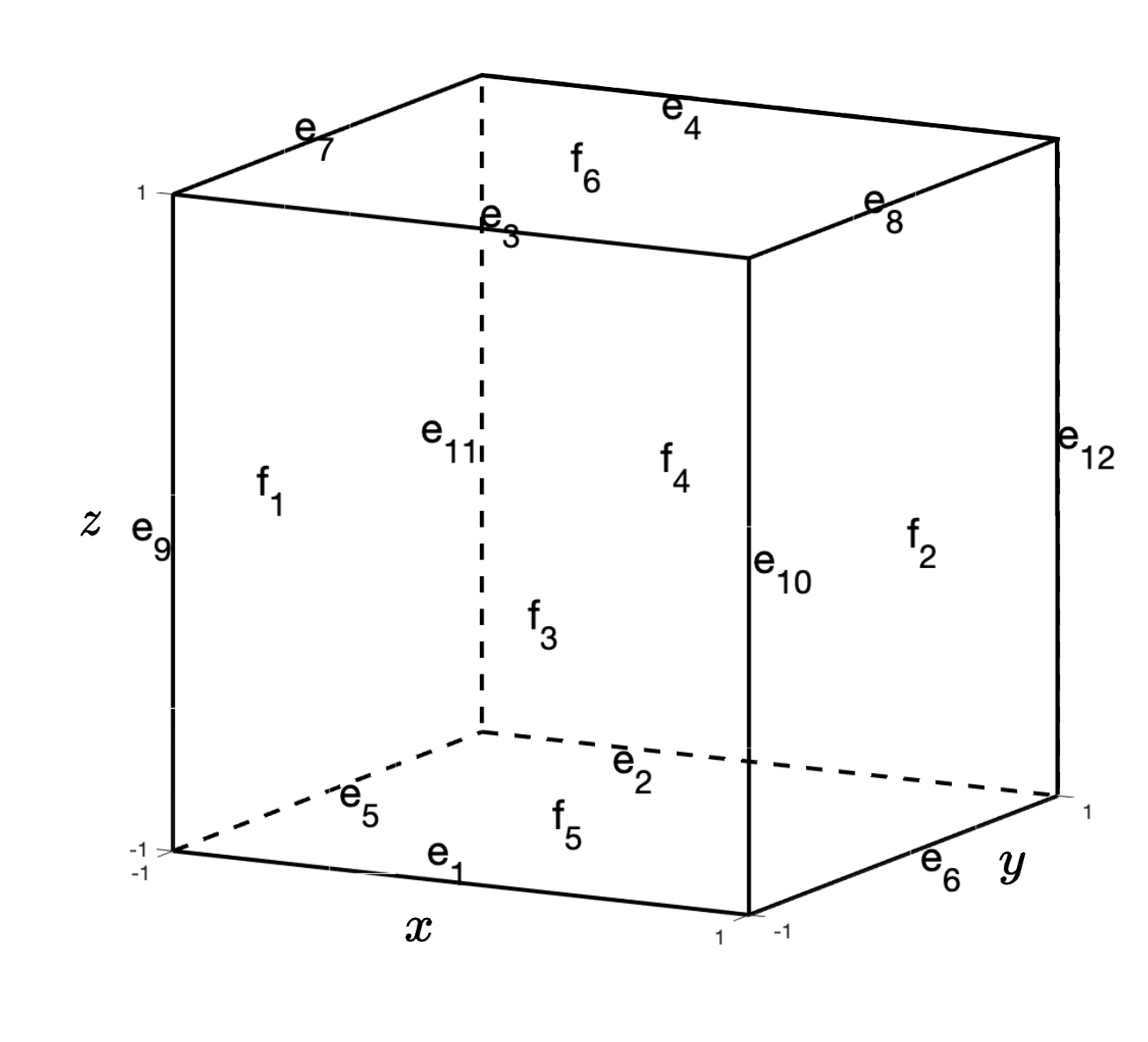}
			\caption{the  element $K$}\label{patch}
		\end{figure}%

	We will prove the lemma only for ${\bm V}^0_h(\mathcal T_h)$, and the lemma for ${\bm V}^1_h(\mathcal T_h)$ can be obtained similarly. For an element $K\in \mathcal T_h$, we sort the edges and faces as in Figure \ref{patch}.
	 Then we can rewrite $\bm v_h\in {\bm V}^0_h(\mathcal T_h)$ and $\bm I_h\bm v_h\in {\bm V}_{h,c}^0(\mathcal T_h)$ as follows,
	\begin{align}\label{vhT}
			\begin{split}
				\bm v_h|_{K}=&\sum_{i=1}^{12}\bm N_i^e\int_{e_i}\bm  v_h\cdot\bm \tau_{e_i}\d s
				+\sum_{i=1}^6\bm N_{i,1}^{f}\int_{f_i}\curl_h \bm  v_h\cdot \bm \tau_{f_i}^{1}\d A\\
             &+\sum_{i=1}^6\bm N_{i,2}^{f}\int_{f_i}\curl_h \bm  v_h\cdot \bm \tau_{f_i}^{2}\d A,
			\end{split}
   \\
   \bm I_h\bm v_h|_{K}=&\sum_{i=1}^{12}\bm N_{c,i}^{e}\int_{e_i}\bm  v_h\cdot\bm \tau_{e_i}\d s,\label{IhvhT}
		\end{align}
   where  $\bm N_i^{e}$, $\bm N_{i,1}^f$ and $\bm N_{i,2}^f$ are the corresponding dual basis functions of ${\bm V}^{0}_h(\mathcal T_h)$ on $e_i$ and $f_i$, and $\bm N_{c,i}^{e}$ is the dual basis function of $\bm V_{h,c}^{0}$ on $e_i$. 
	
	Plugging \eqref{vhT} and \eqref{IhvhT} in $\int_{\partial K}\bm  q\cdot (\bm v_h-\bm I_h\bm v_h)\times \bm n_{\partial K}\d A$, we have
		\begin{align*}
			\int_{\partial K}\bm  q\cdot (\bm v_h-\bm I_h\bm v_h)\times \bm n_{\partial K}\d A&=\sum_{i=1}^{12}a_{i}\int_{e_i}\bm v_h\cdot\bm \tau_{e_i}\d s+\sum_{i=1}^6\sum_{j=1}^2b_{i,j}\int_{f_i}\curl_h\bm v_h\cdot \bm \tau_{f_i}^{j}\d A,
		\end{align*}
		where $a_i=\int_{\partial K}\bm q\cdot\big( (\bm N_i^e-\bm N_{c,i}^e)\times \bm n_{\partial K}\big)\d A$ and $b_{i,j}=\int_{\partial K}\bm q\cdot\big(\bm N_{i,j}^f\times \bm n_{\partial K}\big)\d A$. 
		
		We will show that $a_i=0$ for $i=1,2,\cdots,12$ and $b_{i,j}=0$ for $i=1,2,\cdots,6$ and $j=1,2$. The functions $\hat{\bm N}_i^{e}=B_K^T{\bm N}_i^{e}\circ F_K$, $\hat{\bm N}_{c,i}^{e}=B_K^T{\bm N}_{c,i}^{e}\circ F_K$, and  $\hat{\bm N}_{i,j}^f = \frac{\|B_K\bm\tau_{f_i}^j\|}{\|B_K\bm n_{f_i}\|}B_K^T{\bm N}_{i,j}^f\circ F_K $ are the corresponding basis functions on the reference element $\hat K$. Here we present only $\hat {\bm N}_6^e$, $\hat {\bm N}_{c,6}^e$ and $\hat{\bm N}_{2,j}^f$:
		\begin{align*}
			\hat {\bm N}_6^e=\left(\begin{matrix}
				3x^2y/64 + 3y^3/64 - y/16\\
				3x/16 - 3xy^2/64 - xz/8 - 3x^3/64 + 3y^2z/64 + 3z^3/64 - 3z/16 + 1/8)\\
				y/16 - 3yz^2/64 - (3y^3)/64
			\end{matrix}\right),	\end{align*}
     		\begin{align*}
			\hat {\bm N}_{2,1}^f=\left(\begin{matrix}
				3zx^2/64 + zx/24 - z/64\\
				0\\
				3x/64 - x^2/24 - 3x^3/64 + 1/24\end{matrix}\right),
			\hat {\bm N}_{2,2}^f=\left(\begin{matrix}
				y/64 - yx/24 - 3yx^2/64\\
				3x^3/64 + x^2/24 - 3x/64 - 1/24\\
				0\end{matrix}\right),
					\end{align*}
					\begin{align*}
     \hat {\bm N}_{c,6}^e=\left(\begin{matrix}
				0\\
				(1-z)(1+x)/8\\
				0
			\end{matrix}\right).   \end{align*}
Then
		\begin{align*}
			a_i&=\int_{\partial K}\bm q\cdot\big( (\bm N_i^e-\bm N_{c,i}^e)\times \bm n_{\partial K}\big)\d A
			=\int_{\partial{\hat K}}\hat {\bm q}\cdot\big( ( \hat{\bm N}_i^e-\hat{\bm N}_{c,i}^e)\times \hat{\bm n}_{\partial \hat{K}}\big)\d \hat A,
		\end{align*}
		with ${\bm q}\circ F_K = B_K^{-T}\hat {\bm q}$.
		After a direct calculation, we  derive
		$a_i=0$.
		Similarly, 
				\begin{align*}
			b_{i,j}=\int_{\partial K}\bm q\cdot (\bm N_{i,j}^f\times \bm n_{\partial K})\d A = \int_{\partial{\hat K}}\hat {\bm q}\cdot ( \hat{\bm N}_{i,j}^f\times \hat{\bm n}_{\partial \hat K})\d \hat A=0
		\end{align*}
  with ${\bm q}\circ F_K = \frac{\|B_K\bm\tau_{f_i}^j\|}{\|B_K\bm n_{f_i}\|}B_K^{-T}\hat {\bm q}$. 
	\end{proof}

	Applying Theorem \ref{converg2} and Lemma \ref{jump-property}, we can prove the convergence results for \eqref{prob22}.
	\begin{theorem}\label{errorestimate-uh}
		Let $(\bm{u};p)\in H_0(\grad\curl;\Omega)\times H_0^1(\Omega)$ be the solution of the variational problem \eqref{prob2} and $(\bm{u}_{h};p_{h})\in \mathring{\bm{V}}^{r-1}_{h}(\mathcal T_h)\times \mathring S^r_{h}(\mathcal T_h)$ be the solution of the discrete problem \eqref{prob22}. Suppose that $\bm{u} \in \bm H^{r}(\Omega)$ and  $\curl \bm u\in \bm H^2(\Omega)$, then there holds  $p=p_h=0$ and 
		\begin{equation}\label{err:u:anorm}
			\|\bm u-\bm u_h\|_{\grad\curl,h}\leq Ch\big(h^{r-1}|\bm{u}|_{r}+\alpha^{1/2}|\curl\bm u|_{1}+\alpha^{-1/2}\|\bm f\|+\beta\alpha^{-1/2}\|\bm u\|+(\epsilon^{1/2}+h)|\curl\bm u|_2\big).
		\end{equation}
	\end{theorem}
	\begin{proof}
		We use the right-hand side of \eqref{err-abs} to estimate \eqref{err:u:anorm}. Due to  $\bm R^{r-1}_h \bm u \in \mathring{\bm V}^{r-1}_h(\mathcal T_h)$ and  the error estimates of the interpolation $\bm R^{r-1}_h \bm u$ \eqref{Interp_err_R0}--\eqref{Interp_err_R1}, we have
		\begin{align}\label{int-err-1}
			&\inf_{\bm v_h\in \mathring{\bm V}^{r-1}_h(\mathcal T_h)}\|\bm u-\bm v_h\|_{\grad\curl,h}\le \|\bm u- \bm R_h^{r-1} \bm u\|_{\grad\curl,h}\nonumber\\
			\le& Ch(h^{r-1}|\bm u|_r+h|\curl \bm u|_2+\epsilon^{1/2}|\curl \bm u|_2).
		\end{align}
		To estimate the consistency error term $E_h(\bm u,\bm v_0)$, we estimate each term in $E_h(\bm u,\bm v_0)$. Firstly,
		\begin{align}\label{I1}
				&\bigg|\sum_{f\in \mathcal{F}_h}\epsilon\langle{\partial_{\bm n_f} (\curl\bm u)}, [\![\curl_h\bm v_0]\!]\rangle_f\bigg|\nonumber\\
				=&\bigg|\sum_{f\in \mathcal{F}_h}\epsilon\langle{\partial_{\bm n_f} (\curl\bm u)}, [\![\curl_h\bm v_0]\!]\rangle_f-\sum_{f\in \mathcal{F}_h}\epsilon\langle\mathcal P_{f}{\partial_{\bm n_f} (\curl\bm u)}, [\![\curl_h\bm v_0]\!]\rangle_f\bigg|\nonumber\\
		    =&\epsilon\bigg|\sum_{K\in\mathcal{T}_h}\sum_{f\subset\partial K}\int_f\left( {\partial_{\bm n_{\partial K}} (\curl\bm u)} -\mathcal P_{f}{\partial_{\bm n_{\partial K}}(\curl\bm u)}\right)\cdot\curl\bm v_0\d A\bigg|\nonumber\\
				=&\epsilon\bigg|\sum_{K\in\mathcal{T}_h}\sum_{f\subset\partial K}\int_f\left( {\partial_{\bm n_{\partial K}} (\curl\bm u)} -\mathcal P_{f}{\partial_{\bm n_{\partial K}} (\curl\bm u)}\right)\cdot\left(\curl\bm v_0-\mathcal P_{f}\curl\bm v_0\right)\d A\bigg|\nonumber\\
				\le&C\epsilon h |\curl\bm u|_2|\curl_h \bm v_0|_{1,h}
				\le C\epsilon^{1/2} h |\curl\bm u|_2\|\bm  v_0\|_{a_h}.
		\end{align}
		Here and hereafter $\bm n_{\partial K}$ is unit outward normal vector to $\partial K$ and $\mathcal P_{f}$ is the $L^2$ projection to $P_0(f)$. The projection $\mathcal P_{f}$ satisfies
		\begin{align}\label{faceproj}
		    \|w-\mathcal P_{f}w\|_f\leq \|w-\mathcal P_K w\|_f\leq Ch^{-1/2}\Big(\|w-\mathcal P_{K} w\|_{K}+h|w-\mathcal P_K w|_{1,K}\Big)\leq Ch^{1/2}|w|_{1,K}.
		\end{align}
		To estimate the second term of the consistency error, we apply 
		Lemma \ref{jump-property} to get
		\begin{align}\label{I2}
				&\bigg|\sum_{f\in\mathcal F_h}\langle\alpha\curl\bm u, [\![\bm v_0\times \bm n_f]\!]\rangle_f\bigg|=\alpha\bigg|\sum_{f\in\mathcal F_h}\langle\curl\bm u, [\![(\bm v_0-\bm I_h^{r-1}\bm v_0)\times \bm n_f]\!]\rangle_f\bigg|\nonumber\\
				=&\alpha\bigg| \sum_{K\in\mathcal{T}_h} \int_{\partial K}  (\curl\bm u-\mathcal P_{K}\curl\bm u)\cdot ((\bm v_0-\bm I_h^{r-1}\bm v_0)\times \bm n_{\partial K})\d A\bigg|\nonumber\\
				=&\alpha \sum_{K\in\mathcal{T}_h}\bigg|\int_K  (\curl\bm u-\mathcal P_{K}\curl\bm u)\cdot \curl(\bm v_0-\bm I_h^{r-1}\bm v_0)\d V\bigg|\nonumber\\ 
				 &+\alpha \sum_{K\in\mathcal{T}_h}\bigg|\int_K  \curl(\curl\bm u-\mathcal P_{K}\curl\bm u)\cdot (\bm v_0-\bm I_h^{r-1}\bm v_0)\d V\bigg|\nonumber\\
				\le& C \alpha h |\curl \bm u|_1\|\curl_h\bm v_0\|
			\le C \alpha^{1/2}h |\curl \bm u|_1\|\bm  v_0\|_{a_h},
		\end{align}
			To estimate the last term of the consistency error, we follow an analogous idea in \cite{Huang2020Nonconforming} and have
		\begin{align}\label{I3}
				&\bigg|\sum_{f\in \mathcal{F}_h}\langle \epsilon\Delta\curl\bm u,[\![\bm v_0\times \bm n_f]\!]\rangle_f\bigg|=\bigg|\sum_{f\in \mathcal{F}_h}\langle \epsilon\Delta\curl\bm u,[\![(\bm v_0-\bm I_h\bm v_0)\times \bm n_f]\!]\rangle_f\bigg|\nonumber\\
				=&\bigg|\sum_{K\in \mathcal{T}_h}\langle \epsilon\Delta\curl\bm u,(\bm v_0-\bm I_h\bm v_0)\times \bm n_{\partial K}\rangle_{\partial K}\bigg|\nonumber\\
				\le &\sum_{K\in \mathcal{T}_h}\big|\left(\epsilon\curl\Delta\curl \bm u,(\bm v_0-\bm I_h\bm v_0)\right)_K-\left(\epsilon\Delta\curl \bm u,\curl (\bm v_0-\bm I_h\bm v_0)\right)_K\big|\nonumber\\
				\overset{\eqref{OriginProblem}}=&\sum_{K\in \mathcal{T}_h}\big|\left(\alpha\curl\curl \bm u+\beta\bm u-\bm f,\bm v_0-\bm I_K \bm v_0\right)_K-\left(\epsilon\Delta\curl \bm u,\curl (\bm v_0-\bm I_K \bm v_0)\right)_K\big|\nonumber\\
				\le &\alpha h|\curl \bm u|_1\|\curl_h\bm v_0\|+\beta h\| \bm u\|\|\curl_h\bm v_0\|+Ch\|\bm f\|\|\curl_h\bm v_0\|+Ch\epsilon|\curl \bm u|_2|\curl_h \bm v_0|_{1,h}\nonumber\\
				\le &Ch\left(\alpha^{-1/2}\|\bm f\|+\alpha^{1/2} |\curl \bm u|_1+{\beta}{\alpha^{-1/2}}\| \bm u\|+\epsilon^{1/2}|\curl \bm u|_2\right)\| \bm v_0\|_{a_h}.
		\end{align}
		where we have used \eqref{curlestimate} and \eqref{L2estimate}.

Now we  can obtain \eqref{err:u:anorm} by combining \eqref{int-err-1}, \eqref{I1}, \eqref{I2}, and \eqref{I3}.
	\end{proof}
	
	\section{Uniform Error Estimates}
	In Theorem \ref{errorestimate-uh}, we assume $|\curl\bm u|_2$ are bounded. However, we can not expect  that the bound of $|\curl\bm u|_2$ is independent of $\epsilon$. To reveal this, we consider the following second-order problem:
	\begin{align}\label{red-pro}
		\begin{split}
			\curl^2\widetilde{\bm u}+\widetilde{\bm u}&=\bm f\quad\text{in}\ \Omega,\\
			\div \widetilde{\bm u}&=0\quad\text{in}\ \Omega,\\
			\widetilde{	\bm u}\times \bm n&=0\quad\text{on}\ \partial\Omega.
		\end{split}
	\end{align}
	We assume $\widetilde{\bm u}$ satisfies the following regularity estimate
	\begin{align}\label{regularity-0}
		\|\widetilde{\bm u}\|_1+\|\curl \widetilde{\bm u}\|_1\le C\|\bm f\|.
	\end{align}
Without loss of generality, we assume $\alpha=\beta=1$ in this section. Let $\bm u$ be the solution of problem \eqref{OriginProblem}.
	According to the proof of \cite[Lemma 3.1 -- 3.2]{huang2022robust}, we can prove
	\begin{align}\label{regularity-1}
		&\|\bm u\|+\|\curl \bm u\|+\epsilon^{1/2}|\curl \bm u|_1\le C\|\bm f\|,\\\label{regularity-2}
		&\epsilon|\curl \bm u|_2+\epsilon^{1/2}|\curl \bm u|_1+\|\bm u-\widetilde{\bm u}\|_1\le C\epsilon^{1/4}\|\bm f\|.
	\end{align}
 
 From \eqref{regularity-2}, we can see the bound of $|\curl \bm u|_2$ will blow up as $\epsilon$ approaches 0, in which case the estimate \ref{errorestimate-uh} would fail. In this section, we will provide a uniform error estimate with respect to $\epsilon$. 
	
	Subtracting \eqref{red-pro} from \eqref{OriginProblem} leads to
	\[-\epsilon\curl\Delta\curl \bm u+\curl\curl(\bm u-\widetilde{\bm u})+\bm u-\widetilde{\bm u}=0.\]
	For $\bm v\in H_0(\curl;\Omega)$, by using integration by parts,  we can get
	\begin{align}\label{ax-1}
	(\widetilde{\bm u}-\bm u,\bm v)
	=	-\epsilon(\curl\Delta\curl \bm u,\bm v)+(\curl\curl(\bm u-\widetilde{\bm u}),\bm v)\nonumber\\
	=	-\epsilon(\Delta\curl \bm u,\curl \bm v)+(\curl(\bm u-\widetilde{\bm u}),\curl \bm v).
	\end{align}
	Before we present the uniform error estimate, we first prove some new approximation properties of $\bm R_h^{r-1}$, which is based on the following boundedness of the interpolation operators $\bm{R}^{r-1}_{\hat K}$ and $\bm{\Pi}_{\hat K}$:
	\begin{align}
		&\|\bm R_{\hat K}^{r-1} \hat{\bm v}\|_{\hat K}\le \|\hat{\bm v}\|_{1,\hat K}+\|\curl\hat{\bm v}\|_{\partial \hat K}\le \|\hat{\bm v}\|_{1, \hat K}+\|\curl \hat{\bm v}\|_{\hat K}^{1/2}\|\curl \hat{\bm v}\|_{1,\hat K}^{1/2},\label{boundedness-Rh}\\
		&\|{\bm{\Pi}}_{\hat K}\hat{\bm v}\|_{\hat K}\le C\|\hat{\bm v}\|_{\partial \hat K}\le C\|\hat{\bm v}\|_{ \hat K}^{1/2}\|\hat{\bm v}\|_{1,\hat  K}^{1/2}.\label{boundedness-Pih}
	\end{align}
Here we have used the trace inequality \cite[Theorem 1.6.6]{brenner2008mathematical}
\begin{align}\label{traceinq2}
    \|\hat u\|_{\partial \hat K}\leq \|\hat u\|^{1/2}_{\hat K}\|\hat u\|_{1,\hat K}^{1/2},
\end{align}
and the first inequality of \eqref{boundedness-Rh} is a direct result of the proof of Lemma \ref{boundednessofDOF}. 
	\begin{lemma}
	    		If $\bm{u}\in \bm H^{1}(K)$, $\curl\bm{u}\in \bm H^{2}(K)$, there hold the following error estimates
		\begin{align}
			\|\bm{u}-\bm{R}^{r-1}_K\bm{u}\|_K&\leq Ch(\|\bm{u}\|_{1,K}+\|\curl\bm u\|^{1/2}_K\|\curl \bm u\|_{1,K}^{1/2}),\label{Interp_err_R0-1}\\
			\|\curl(\bm{u}-\bm{R}^{r-1}_K\bm{u})\|_{j,K}&\leq Ch^{1/2}\|\curl\bm{u}\|_{j,K}^{1/2}|\curl\bm{u}|_{j+1,K}^{1/2},\ j=0,1.\label{Interp_err_R1-1}
		\end{align}
	\end{lemma}
	\begin{proof}

 Again, we follow the proof of \cite[Theorem 5.41]{Monk2003} and use the transformation \eqref{tran} to have
		\begin{align*}
			\|\bm u-\bm R_{ K}^{r-1} \bm u\|_K^2&\leq h\inf_{\hat{\bm p}\in \bm P_0(\hat K)}\int_{\hat K}\left((\hat{\bm u}+\hat{\bm p})-\bm R_{\hat K}^{r-1}( \hat{\bm u}+\hat{\bm p})\right)^2\d\hat  V.
		\end{align*}
We then apply	\eqref{boundedness-Rh} and \cite[Theorem 5.5]{Monk2003} to derive
\begin{align*}
\|\bm u-\bm R_{ K}^{r-1} \bm u\|_K^2	&\le h\inf_{\hat{\bm p}\in \bm P_0(\hat K)}\left(\|\hat {\bm u}+\hat{\bm p}\|_{1,\hat K}^2+\|\widehat{\curl}(\hat {\bm u}+\hat{\bm p})\|_{\hat K}\| \widehat{\curl}(\hat {\bm u}+\hat{\bm p})\|_{1,\hat K}\right)\\
			&=h\left(\inf_{\hat{\bm p}\in \bm P_0(\hat K)}\|\hat {\bm u}+\hat{\bm p}\|_{1,\hat K}^2+\|\widehat{\curl}\hat {\bm u}\|_{\hat K}\| \widehat{\curl}\hat {\bm u}\|_{1,\hat K}\right)\\
			&\le h\left(|\hat {\bm u}|_{1,\hat K}^2+\|\widehat{\curl}\hat {\bm u}\|_{\hat K}\| \widehat{\curl}\hat {\bm u}\|_{1,\hat K}\right)\\
			&\le h^2\left(| {\bm u}|_{1,K}^2+\| \curl{\bm u}\|_K\| \curl {\bm u}\|_{1,K}\right).
\end{align*}
Similarly, we can prove \eqref{Interp_err_R1-1}  by \eqref{boundedness-Pih} for $j=0,1$. 

	\end{proof}
\begin{lemma} Given $K\in \mathcal T_h$, let $f\in \mathcal F_h(K)$. The $L^2$ projection $\mathcal P_{f}$ has the following estimate:
\[\|w - \mathcal P_{f}w\|_f\leq C\|w\|_K^{1/2}\|w\|_{1,K}^{1/2}, \quad w\in H^1(K).\]
\end{lemma}
\begin{proof}\label{p0f-bound}
For $w\in H_1(K)$, we define $\hat w = w\circ F_K$.
By scaling argument, \eqref{traceinq2}, and \cite[Theorem 5.5]{Monk2003}, we have
\begin{align*}
&\|w - \mathcal P_{f}w\|^2_f\leq h^2\| \hat w - \mathcal P_{\hat f}\hat w\|^2_{\hat f}\leq Ch^2\|\hat w - \mathcal P_{\hat K}\hat w\|^2_{\hat f} = Ch^2\inf_{\hat q\in P_0(\hat K)}\|(I - \mathcal P_{\hat K})(\hat w+\hat q)\|^2_{\hat f}\\
\leq &Ch^2\inf_{\hat q\in P_0(\hat K)}(\|\hat w+\hat q\|_{\hat f}^2+\|\hat w+\hat q\|_{\hat K}^2)\leq Ch^2\inf_{\hat q\in P_0(\hat K)}\|\hat w+\hat q\|_{\hat K}\|\hat w+\hat q\|_{1,\hat K} \\
\leq  &Ch^2\inf_{\hat q\in P_0(\hat K)}\|\hat w+\hat q\|_{\hat K}\inf_{\hat q\in P_0(\hat K)}\|\hat w+\hat q\|_{1,\hat K}
\leq  Ch^2\|\hat w\|_{\hat K}|\hat w|_{1,\hat K}\leq C\| w\|_{ K}| w|_{1, K}.
\end{align*}

\end{proof}
	\begin{theorem}
		Let $(\bm{u};p)\in H_0(\grad\curl;\Omega)\times H_0^1(\Omega)$ be the solution of the variational problem \eqref{prob2} and $(\bm{u}_{h};p_{h})\in \mathring{\bm{V}}^{r-1}_{h}(\mathcal T_h)\times \mathring S^r_{h}(\mathcal T_h)$ be the solution of the discrete problem \eqref{prob22}. Suppose  $\bm{f} \in \bm L^{2}(\Omega)$, then under assumption \eqref{regularity-0}, there holds  $p=p_h=0$ and 
		\begin{equation}\label{err:u:anorm-1}
			\|\bm u-\bm u_h\|_{\grad\curl,h}\leq Ch^{1/2}\|\bm f\|.
		\end{equation}
	\end{theorem}
	\begin{proof}
		We use \eqref{err-abs} to estimate \eqref{err:u:anorm-1}. By \eqref{Interp_err_R0-1}--\eqref{Interp_err_R1-1}, \eqref{Interp_err_R1}, and \eqref{regularity-0}--\eqref{regularity-2},  we have
		\begin{align}\label{int-err-2}
				&\inf_{\bm v_h\in  \mathring{\bm V}^{r-1}_h(\mathcal T_h)}\|\bm u-\bm v_h\|_{\grad\curl,h}\le \|\bm u- \bm R_h^{r-1}  \bm u\|_{\grad\curl,h}\nonumber\\
				\le&\epsilon^{1/2}\|\grad\curl (\bm u- \bm R_h^{r-1} \bm u)\|+\big\|\curl \big((\bm u-\widetilde{\bm u})-  \bm R_h^{r-1}   (\bm u-\widetilde{\bm u})\big)\big\|\nonumber\\&+\|\curl(\widetilde{\bm u}- \bm R_h^{r-1}  \widetilde{\bm u})\|+\|(\bm u-\widetilde{\bm u})- \bm R_h^{r-1}  (\bm u-\widetilde{\bm u})\|+\|\widetilde{\bm u}-  \bm R_h^{r-1}  \widetilde{\bm u}\|\nonumber\\
				\le&  Ch^{1/2}\Big(\epsilon^{1/2}|\curl \bm u|_2^{1/2}\|\curl \bm u\|_1^{1/2}+|\curl( \bm u-\widetilde{\bm u})|_1^{1/2}\|\curl (\bm u-\widetilde{\bm u})\|^{1/2}\nonumber\\&+h^{1/2}\big(\|\curl\widetilde{\bm u} \|_1+\|\bm u-\widetilde{\bm u} \|_1+\|\curl(\bm u-\widetilde{\bm u})\|^{1/2}\|\curl(\bm u-\widetilde{\bm u})\|_1^{1/2}\nonumber\\&+\|\widetilde{\bm u} \|_1+\|\curl\widetilde{\bm u}\|^{1/2}\|\curl\widetilde{\bm u}\|_1^{1/2}\big)\Big)\nonumber\\
				\le&  Ch^{1/2}\big(\epsilon^{1/2}|\curl \bm u|_2^{1/2}\|\curl \bm u\|_1^{1/2}+\|\curl( \bm u-\widetilde{\bm u})\|_1^{1/2}\|\bm u-\widetilde{\bm u}\|_1^{1/2}+\|\bm f\|\big)\nonumber\\
					\le& Ch^{1/2}\big(\epsilon^{1/8}|\curl\bm u|_1^{1/2}\|\bm f\|^{1/2}+\|\bm f\|\big)
				\le Ch^{1/2}\|\bm f\|.
		\end{align}
	
		Now we estimate the consistency error term by term.  We proceed as in \eqref{I1} to estimate the first term.  Applying \eqref{faceproj} to $\curl_h\bm v_0$ and Lemma \ref{p0f-bound} to $\partial_{\bm n_{\partial K}} (\curl\bm u)$, respectively, it holds that 
		\begin{align}\label{I1-1}
				&\left|\sum_{f\in \mathcal{F}_h}\langle\epsilon{\partial_{\bm n_f} (\curl\bm u)}, [\![\curl_h\bm v_0]\!]\rangle_f\right|\nonumber\\
				=&\epsilon\bigg|\sum_{K\in\mathcal{T}_h}\sum_{f\subset{\partial K}}\int_f\left( {\partial_{\bm n_{\partial K}} (\curl\bm u)} -\mathcal P_{f}{\partial_{\bm n_{\partial K}} (\curl\bm u)}\right)\cdot\left(\curl_h\bm v_0-\mathcal P_{f}\curl\bm v_0\right)\d A\bigg|\nonumber\\
				\le&C\epsilon h^{1/2}|\curl\bm u|_1^{1/2}|\curl \bm u|_2^{1/2} |\curl_h \bm v_0|_{1,h}\nonumber\\
				\le &C\epsilon^{1/2} h^{1/2} |\curl\bm u|_1^{1/2}|\curl \bm u|_2^{1/2} \|\bm  v_0\|_{a_h}\nonumber\\
				\le&Ch^{1/2}\|\bm f\| \|\bm  v_0\|_{a_h}.
		\end{align}
	For the last two terms of the consistency error, we have
		\begin{align}\label{I1-2}
				&\sum_{f\in\mathcal F_h}\langle\curl\bm u, [\![\bm v_0\times \bm n_f]\!]\rangle_f-\sum_{f\in \mathcal{F}_h}\epsilon\langle \Delta\curl\bm u,[\![\bm v_0\times \bm n_f]\!]\rangle_f\nonumber\\
				=&\sum_{f\in\mathcal F_h}\big(\langle\curl\bm u-\curl\widetilde{\bm  u}, [\![\bm v_0\times \bm n_f]\!]\rangle_f+\langle\curl\widetilde{\bm  u}, [\![\bm v_0\times \bm n_f]\!]\rangle_f
				-\epsilon\langle \Delta\curl\bm u,[\![\bm v_0\times \bm n_f]\!]\rangle_f\big)\nonumber\\
				=&(\curl  \curl (\bm u-\widetilde{\bm  u}),  \bm v_0)- (\curl (\bm u-\widetilde{\bm  u}), \curl_h \bm v_0)+\langle\curl\widetilde{\bm  u}, [\![\bm v_0\times \bm n_f]\!]\rangle_f
			\nonumber	\\&-\epsilon(\curl \Delta\curl \bm u, \bm v_0)+\epsilon( \Delta\curl \bm u, \curl_h \bm v_0)\qquad\text{(integration by parts)}\nonumber\\
				\overset{\eqref{ax-1}}=&\epsilon( \Delta\curl \bm u, \curl_h( \bm v_0-\bm I_h\bm v_0))- (\curl (\bm u-\widetilde{\bm  u}), \curl_h( \bm v_0-\bm I_h\bm v_0))\nonumber\\
				&-(\bm u-\widetilde{\bm u},\bm v_0-\bm I_h\bm v_0)+\langle\curl\widetilde{\bm  u}, [\![\bm v_0\times \bm n_f]\!]\rangle_f\nonumber\\
				\leq&Ch^{1/2}|\curl_h \bm v_0|_1^{1/2}\|\curl_h \bm v_0\|^{1/2}\big(\epsilon|\curl \bm u|_{2}+\|\curl (\bm u-\widetilde{\bm  u})\|\big)\nonumber\\
				&+Ch\|\bm u-\widetilde{\bm u}\|\|\curl_h\bm v_0\|+\langle\curl\widetilde{\bm  u}, [\![\bm v_0\times \bm n_f]\!]\rangle_f\quad \big(\eqref{curlestimate} \text{ and } \eqref{L2estimate}\big)\nonumber\\
				\overset{\eqref{regularity-2}}\le&Ch^{1/2}\epsilon^{1/4}\|\bm f\||\curl_h \bm v_0|_{1,h}^{1/2}\|\curl_h \bm v_0\|^{1/2}+Ch\epsilon^{1/4}\|\bm f\|\|\curl_h\bm v_0\|+\langle\curl\widetilde{\bm  u}, [\![\bm v_0\times \bm n_f]\!]\rangle_f\nonumber\\
				\overset{\eqref{I2}}\le& Ch^{1/2}\|\bm f\|\|\bm v_0\|_{a_h}+Ch\epsilon^{1/4}\|\bm f\|\|\bm v_0\|_{a_h}+Ch|\curl\widetilde{\bm u}|_1\|\bm v_0\|_{a_h}.
		\end{align}
		Combining  \eqref{int-err-2}, \eqref{I1-1}, and  \eqref{I1-2}, we can get \eqref{err:u:anorm-1} based on the assumption \eqref{regularity-0}.
	\end{proof}
\section{NUMERICAL EXPERIMENTS}
In this section, numerical results are provided to verify the theoretical results. In this section, we assume $\alpha=\beta=1$. 

\begin{example}
\end{example}
We first consider problem \eqref{OriginProblem} on $\Omega=[0,1]^3$ with a smooth exact solution
\begin{align*}
\bm u =&\left(
\begin{array}{c}
\sin^3(\pi x)\sin^2(\pi y)\sin^2(\pi z)\cos(\pi y)\cos(\pi z) \\
\sin^3(\pi y)\sin^2(\pi z)\sin^2(\pi x)\cos(\pi z)\cos(\pi x)\\
-2\sin^3(\pi z)\sin^2(\pi x)\sin^2(\pi y)\cos(\pi x)\cos(\pi y)
\end{array}
\right).
\end{align*}
The source term $\bm f=-\epsilon\text{curl}\Delta\text{curl}\bm u+\text{curlcurl} \bm u +\bm u $ can be obtained by a simple  calculation. 

We first uniformly partition the domain into eight small cubes and then each small cube is partitioned into eight half-sized cubes to get a refined grid.

We first use the finite element spaces $\bm V_h^0(\mathcal T_h)\times S^1_h(\mathcal T_h)$ in discrete problem \eqref{prob22}. We present $\|\bm u-\bm u_h\|$ and $\|\bm u-\bm u_h\|_{\grad\curl,h}$ for different values of $\epsilon$ in Figure \ref{fig1}. From the result of $\|\bm u-\bm u_h\|_{\grad\curl,h}$ shown in Figure \ref{fig1}(a),  we can observe that the numerical results are consistent with Theorem \ref{errorestimate-uh}. 

We then apply  $\bm V_h^1(\mathcal T_h)\times S^2_h(\mathcal T_h)$ in \eqref{prob22}. See Figure \ref{fig2} for the numerical results. Again the numerical result shown in Figure \ref{fig2}(a) validate the theoretical result in  Theorem \ref{errorestimate-uh}. 

We also observe from Figures \ref{fig1}(a) and \ref{fig2}(a) that when $
\epsilon = 10^{-4}, 10^{-6}$, and $10^{-8}$, the convergence rate of $\|\bm u-\bm u_h\|_{\grad\curl,h}$ is 2. This is because the part $\|\curl(\bm u-\bm u_h)\|$ dominates in these cases, and $\|\curl(\bm u-\bm u_h)\|$ has convergence rate 2.

From Figures \ref{fig1}(b) and  \ref{fig2}(b), we can see that the difference between the two families $\bm V_h^0(\mathcal T_h)$ and $\bm V_h^1(\mathcal T_h)$ is the convergence order in the sense of $L^2$ norm.

\begin{figure}[h]
  \centering
  \subfigure[$\|\bm u -\bm u_h\|_{a_h}$]{
   \begin{minipage}[t]{0.42\linewidth}
    \centering
    \includegraphics[width=2in]{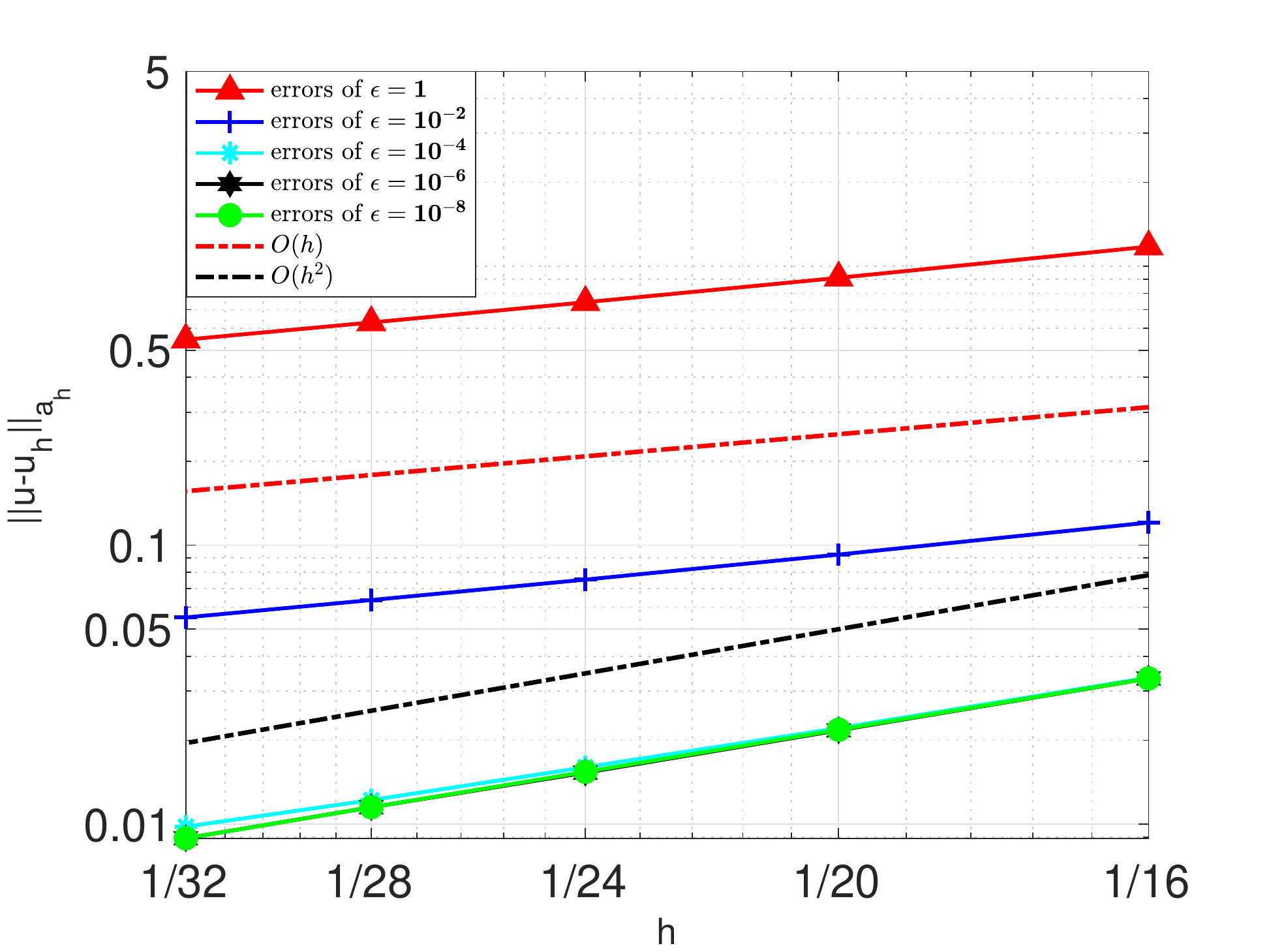}
   \end{minipage}
    \label{fig1b}
  }
  \subfigure[$\|\bm u -\bm u_h\|$]{
   \begin{minipage}[t]{0.42\linewidth}
    \centering
    \includegraphics[width=2in]{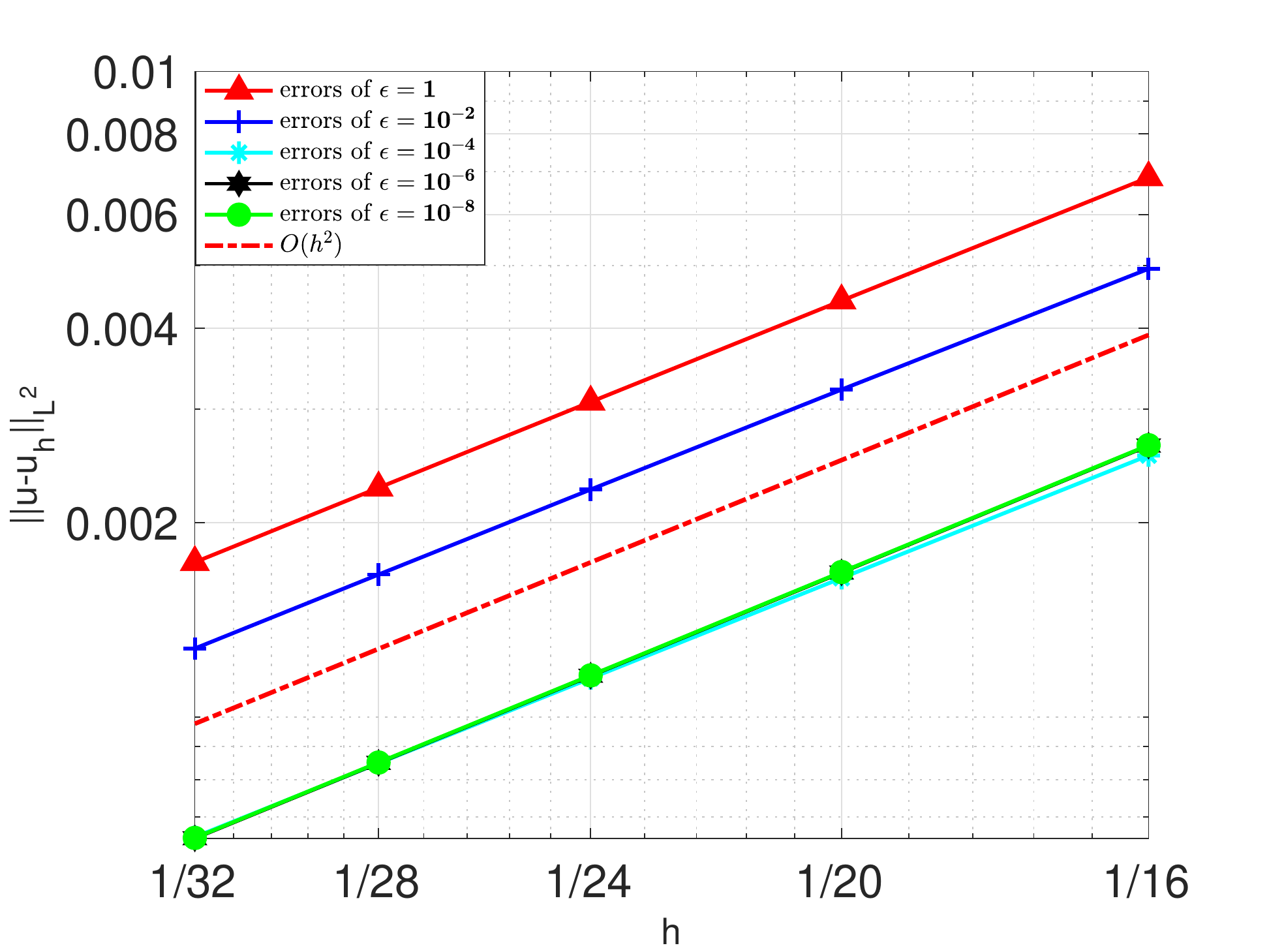}
   \end{minipage}
    \label{fig1a}
  }
  \caption{Error between $\bm u-\bm u_h$ with $r=1$.}
  \label{fig1}
 \end{figure}

 \begin{figure}[h]
  \centering
  \subfigure[$\|\bm u -\bm u_h\|_{a_h}$]{
   \begin{minipage}[t]{0.42\linewidth}
    \centering
    \includegraphics[width=2in]{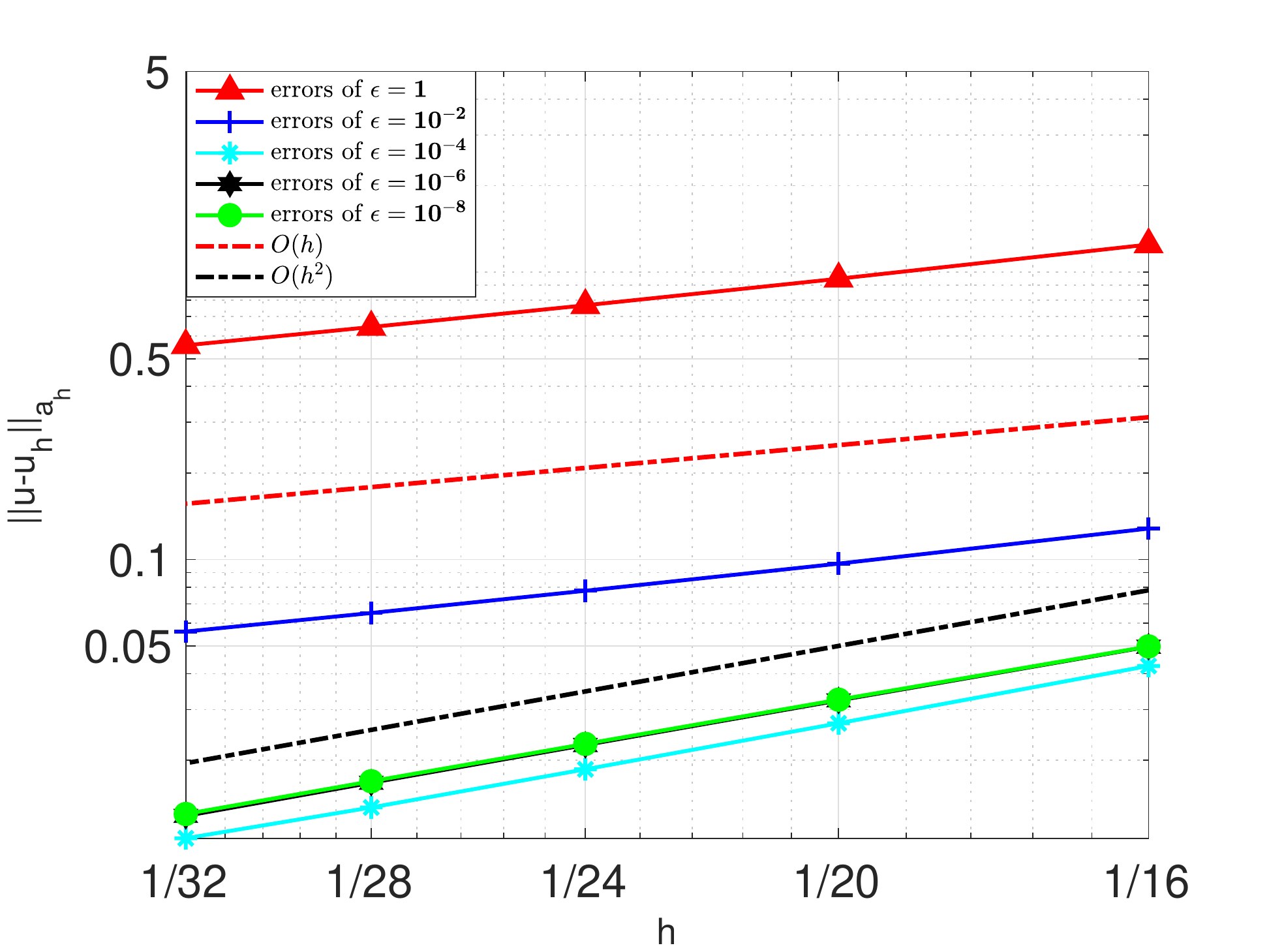}
    \label{fig2b}
   \end{minipage}
  }
  \subfigure[$\|\bm u -\bm u_h\|$]{
   \begin{minipage}[t]{0.42\linewidth}
    \centering
    \includegraphics[width=2in]{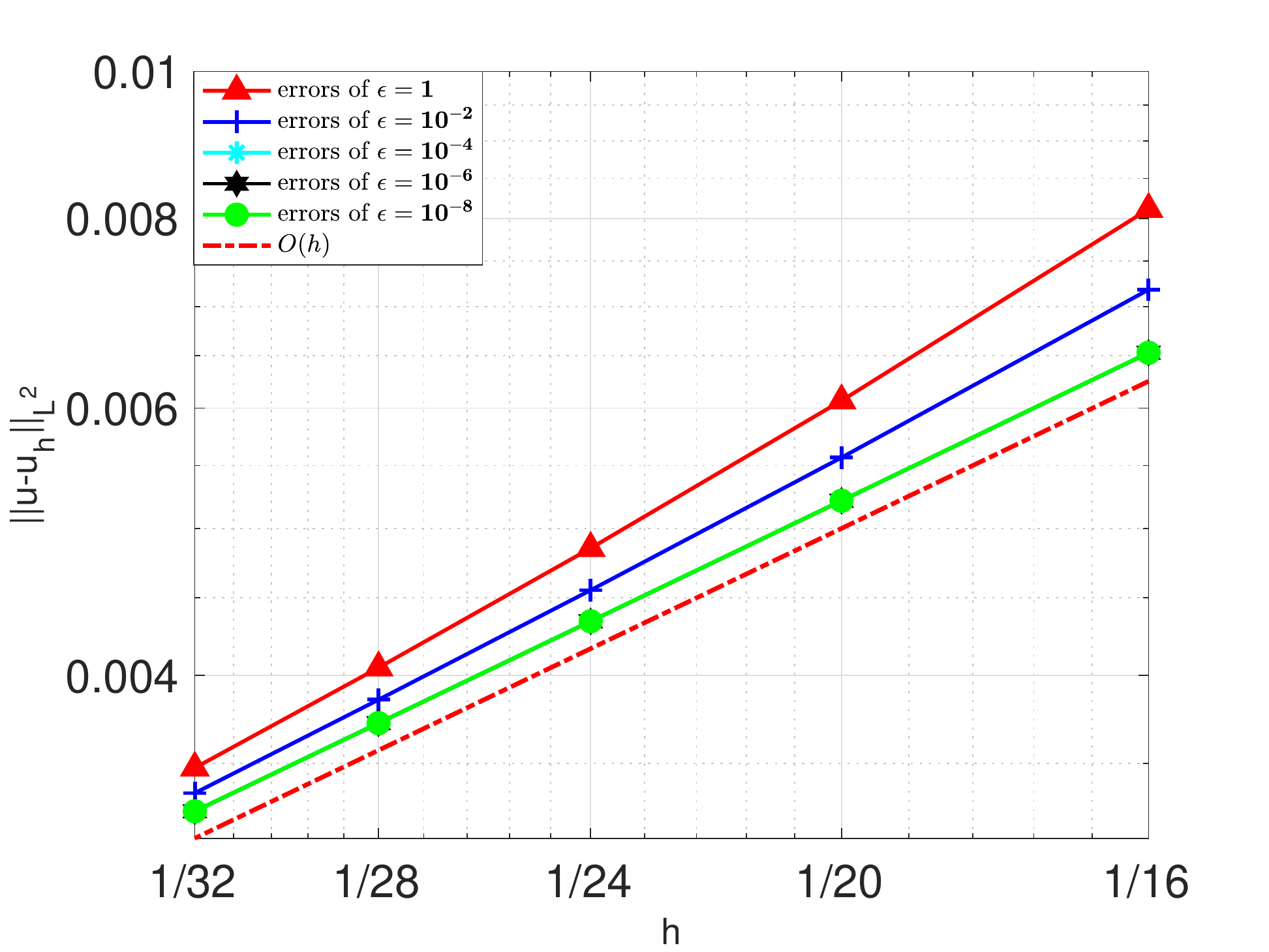}
    \label{fig2a}
   \end{minipage}
  }
  \caption{Error between $\bm u-\bm u_h$ with $r=2$.}
  \label{fig2}
 \end{figure}

\begin{example}\label{ex1}
\end{example}
We now consider an exact solution with boundary layer.  We take the source term $\bm f$ as
\[\bm f=\text{curlcurl}\widetilde{\bm{u}} +\widetilde{\bm{u}}\]
with
\begin{equation*}
\widetilde{\bm{u}}=\begin{pmatrix}
2\pi \cos(\pi y)\sin(\pi x)^2\sin(\pi y)\sin(\pi z)^2\\
-2\pi\cos(\pi x)\sin(\pi x)\sin(\pi y)^2\sin(\pi z)^2\\
0
\end{pmatrix}.
\end{equation*} 
Then according to \eqref{regularity-2},  $\|\bm u-\widetilde{\bm{u}}\|_1\leq C\epsilon^{1/4}\|\bm f\|$, which goes to 0 when $\epsilon$ approaches 0.

\begin{table}[h]
	\caption{Example \ref{ex1}: errors between exact solution and numerical solution with $\epsilon=10^{-8}$.}
	\begin{tabular}{cccccccccc}
		\hline
		$h$ & $\|\widetilde{\bm u}-\bm u_h\|$ & order & $\|\text{curl}_h(\widetilde{\bm u}- \bm u_h)\|$ & order & $\|\widetilde{\bm u}-\bm u_h\|_{a_h}$ & order \\ \hline
		1/4   & 3.231e$-$01         & *     & 3.451e+00                       & *     & 3.781e+00                & *     \\
		1/8   & 1.623e$-$01            & 0.99     & 1.937e+00                        & 0.83     & 2.108e+00                & 0.84     \\
		1/16  & 8.261e$-$02           & 0.97     & 1.295e+00                       & 0.58     & 1.389e+00                & 0.60     \\
		1/32  & 4.168e$-$02           & 0.99     & 9.049e$-$01                         & 0.52     & 9.628ee$-$01                & 0.53     \\ \hline
	\end{tabular}
\label{table1}
\end{table}

\begin{table}[h]
	\caption{Example \ref{ex1}: errors between exact solution and numerical solution with $\epsilon=0$.}
	\begin{tabular}{cccccccc}
		\hline
		$h$ & $\|\widetilde{\bm u}-\bm u_h\|$ & order & $\|\text{curl}_h(\widetilde{\bm u}- \bm u_h)\|$ & order & $\|\widetilde{\bm u}-\bm u_h\|_{a_h}$ & order \\ \hline
		1/4   & 3.231e$-$01         & *     & 3.451e+00                       & *     & 3.774e+00                & *     \\
		1/8   & 1.623e$-$01            & 0.99     & 1.937e+00                        & 0.83     & 2.100e+00                & 0.85     \\
		1/16  & 8.261e$-$02           & 0.97     & 1.295e+00                       & 0.58     & 1.378e+00                & 0.61     \\
		1/32  & 4.168e$-$02         & 0.99     & 9.049e$-$01                      & 0.52     & 9.466e$-$01               & 0.54     \\ \hline
	\end{tabular}
\label{table2}
\end{table}
We use the finite element space $\bm V_h^0(\mathcal T_h)\times S^1_h(\mathcal T_h)$ again. We present the errors $\bm u-\bm u_h$ and $\text{curl}\bm{u}-\curl\bm u_h$ in the sense of $L^2$ norm, and also the error $\bm u-\bm u_h$ in the energy norm $\|\cdot\|_{a_h}$ in the table \ref{table1}--\ref{table2}. The numerical results show that the errors converge to zeros with order $1$, $0.5$, and $0.5$, which agrees with the theoretical estimates \eqref{err:u:anorm-1}.
\section{Conclusion}
In this paper, we constructed two fully $H(\grad\curl)$-nonconforming finite elements on cubical meshes, which together with the Lagrange element and the nonconforming Stokes element in \cite{zhang2009low}  form a nonconforming finite element Stokes complex. Moreover, we proved a special property of the proposed elements, see Lemma \ref{jump-property}. With this property, we proved the optimal convergence and uniform convergence when applying the elements to solve the singularly perturbed quad-curl problem. 
 
	\bibliographystyle{plain}
	\bibliography{Nonconforming_ref.bib}

\end{document}